\documentclass[reqno]{amsart}

\usepackage{enumerate,amsmath,amssymb,amscd}
\usepackage{color}     
\usepackage{tikz}
\usepackage{hyperref}
\usepackage{MnSymbol}
\usepackage[capitalize]{cleveref}

\theoremstyle{plain}      
 
\newtheorem{thm}{Theorem}[section]     
\newtheorem{theorem}[thm]{Theorem}     
     
\newtheorem{corollary}[thm]{Corollary}     
     
\newtheorem{lemma}[thm]{Lemma}     
     
\newtheorem{proposition}[thm]{Proposition}     
     
\theoremstyle{remark}      
\newtheorem*{rem*}{Remark}
\newtheorem*{ex*}{Example}
\newtheorem*{theorem*}{{\bf Theorem}}
\theoremstyle{definition}

\newtheorem{example}[thm]{Example} 
\newtheorem{remark}[thm]{Remark}

\def\al{{\alpha}}

\def\om{{\omega}}

\def\ga{{\gamma}}
\def\epsilon{{\varepsilon}}
\def\ep{{\varepsilon}}

\def\phi{{\varphi}}

\def\diff{\mbox{\sl Diff}}

\DeclareMathAlphabet{\doba}{U}{msb}{m}{n}

\gdef\mR{\doba{R}}

\gdef\mZ{\doba{Z}}

\def\Spin{{\mathop{\rm Spin}}}

\def\tr{{\mathop{\rm Tr}}}

\let\na\nabla     
\let\<\langle 
\let\>\rangle 

\let\witi\widetilde
\def\End{{\mathop{\rm End}}}
\def\Aut{{\mathop{\rm Aut}}}
\def\Hom{{\mathop{\rm Hom}}}
\def\Imm{{\mathop{\rm Imm}}}

\newcommand{\definedas}{\mathrel{\raise.095ex\hbox{\rm :}\mkern-5.2mu=}}
\newcounter{mnotecount}[section]

\newcommand{\U}{\mr{U}}
\newcommand{\SU}{\mr{SU}}
\newcommand{\SO}{\mr{SO}}
\newcommand{\Sp}{\mr{Sp}}
\newcommand{\GL}{\mr{GL}}

\newcommand{\id}{\text{Id}}
\newcommand{\del}{\bar {\partial}}

\newcommand{\mr}{\mathrm}

\newcommand{\mc}{\mathcal}

\newcommand{\R}{\mathbb{R}}
\newcommand{\Z}{\mathbb{Z}}
\newcommand{\C}{\mathbb{C}}

\newcommand{\Id}{\mr{Id}}

\newcommand{\beq}{\begin{equation}}
\newcommand{\ben}{\begin{equation}\nonumber}
\newcommand{\ee}{\end{equation}}

\newcommand{\grad}{\mr{grad}\,}
\newcommand{\Ric}{\mr{Ric}\,}

\renewcommand{\div}{\mathop{\mathrm{div}}}
\let\divv\div
\newcommand{\arsinh}{\operatorname{arsinh}}
\newenvironment{lemproof}{\begin{trivlist}\item[]{\it Proof.}}
 {\end{trivlist}}

\long\def\wegwegenreferee#1{}

\begin{document}     
\title{The spinorial energy functional on surfaces}
 
\author{Bernd Ammann} 
\address{Fakult\"at f\"ur  Mathematik \\
Universit\"at Regensburg \\ Universit\"atsstra{\ss}e 40 \\
D--93040 Regensburg \\  Germany}
\email{bernd.ammann@mathematik.uni-regensburg.de}

\author{Hartmut Wei\ss{}} 
\address{Mathematisches Seminar der Universit\"at Kiel\\ Ludewig-Meyn Stra{\ss}e 4\\ D--24098 Kiel\\ Germany}
\email{weiss@math.uni-kiel.de}

\author{Frederik Witt} 
\address{Institut f\"ur Geometrie und Topologie der Universit\"at Stuttgart\\ Pfaffenwaldring 57\\ D--70569 Stuttgart\\ Germany}
\email{frederik.witt@mathematik.uni-stuttgart.de}

\begin{abstract}
This is a companion paper to \cite{amwewi12} where we introduced the spinorial energy functional and studied its main properties in dimensions equal or greater than three. In this article we focus on the surface case. A salient feature here is the scale invariance of the functional which leads to a plenitude of critical points. Moreover, via the spinorial Weierstra{\ss} representation it relates to the Willmore energy of periodic immersions of surfaces into $\R^3$.
\end{abstract}

\maketitle
 
%
%
%
\section{Introduction}
Let $M^n$ be a closed spin manifold of dimension $n$ with a fixed spin structure $\sigma$. If~$g$ is a Riemannian metric on $M$, we denote by $\Sigma_gM\to M$ the associated spinor bundle. The spinor bundles for all possible choices of $g$ may be assembled into a single fiber bundle $\Sigma M\to M$, the so-called {\em universal spinor bundle}. A section $\Phi\in\Gamma(\Sigma M)$ determines a Riemannian metric $g=g_\Phi$ and a $g$-spinor $\phi=\phi_\Phi\in\Gamma(\Sigma_gM)$ and vice versa. In particular, one can split the tangent space of $\Sigma M$ at $(g_x,\phi_x)$ into a ``horizontal part'' $\odot^2T^*_x\!M$ and a ``vertical'' part $(\Sigma_gM)_x$ (see~\cite{amwewi12} for further explanation). Furthermore, let $S(\Sigma M)$ denote the universal bundle of unit spinors, i.e.\ $S(\Sigma M)=\{\Phi \in \Sigma M\mid|\Phi|= 1 \}$, and $\mc{N} = \Gamma(S(\Sigma M))$ its space of {\em smooth} sections. If we identify $\Phi$ with the pair $(g,\phi)$ we can consider the {\em spinorial energy functional}
\ben
\mc{E}:\mc{N}\rightarrow\R_{\geq 0},\quad(g,\phi)\mapsto\tfrac{1}{2}\int_M|\nabla^g\phi|_g^2\,dv^g
\ee
introduced in~\cite{amwewi12}. Here, $\nabla^g$ denotes the Levi-Civita connection, $|\cdot|_g$ the pointwise norm on spinors in $\Sigma_gM$, and integration is performed with respect to the associated Riemannian volume form $dv^g$. The functional is invariant under the $\Z_2$-extension of the spin diffeomorphism group and rescales as
\beq\label{rescaling}
\mc{E}(c^2g,\phi)=c^{n-2}\mc{E}(g,\phi)
\ee
under homothetic change of the metric by $c>0$. The negative gradient of $\mc{E}$ can be viewed as a map
\beq\label{gradmap}
Q:\mc N\to T\mc N,\quad\Phi\in\mc N \mapsto\bigl(Q_1(\Phi),Q_2(\Phi)\bigr)\in\Gamma(\odot^2T^*\!M) \times \Gamma(\phi^{\perp_g})
\ee
(for a curve $\phi_t$ with $|\phi_t|=1$, $\dot\phi$ must be pointwise perpendicular to $\phi$). In~\cite{amwewi12} we showed that 
for $\Phi = (g,\phi)\in \mc{N}$ we have
\beq\label{neggrad}
\begin{array}{l}
Q_1(\Phi)=-\tfrac{1}{4}|\nabla^g\phi|^2_gg-\tfrac{1}{4}\div_g T_{g,\phi}+\tfrac{1}{2}\langle\nabla^g\phi\otimes\nabla^g\phi\rangle,\\[5pt]
Q_2(\Phi)=-\nabla^{g\ast}\nabla^g\phi+|\nabla^g\phi|^2_g\phi,
\end{array}
\ee
where $T_{g,\phi}\in\Gamma(T^*\!M\otimes\odot^2T^*\!M)$ is the symmetrisation 
in the second and third component of the $(3,0)$-tensor defined by $\<(X\wedge Y)\cdot\phi,\nabla^g_Z\phi\>$ for $X$, $Y$ and $Z$ in $\Gamma(TM)$. Further, $\langle\nabla^g\phi\otimes\nabla^g\phi\rangle$ is the symmetric $2$-tensor defined by $\langle\nabla^g\phi\otimes\nabla^g\phi\rangle(X,Y)=\langle\nabla^g_X\phi,\nabla^g_Y\phi\rangle$.

As a corollary, the critical points for $n\geq3$ are precisely the pairs $(g,\varphi)$ satisfying $\nabla^g\phi=0$, i.e.\ the parallel (unit) spinors. In particular, $g$ must be Ricci-flat and $(g,\phi)$ is an absolute minimiser.

\medskip

The present work investigates the spinorial energy functional on spin surfaces $(M_\gamma,\sigma)$ where $M_\gamma$ is a connected, closed surface of genus $\gamma$ endowed with a fixed spin structure $\sigma$. This differs from the generic case of dimension $n\geq3$ in several aspects. First, the functional is invariant under rescaling by \cref{rescaling}, which leads to a potentially richer critical point structure in two dimensions. Indeed, we will construct in \cref{nonmin.crit.pts} certain flat 2-tori with non-minimising critical points which are saddle points in the sense that the Hessian of the functional is indefinite. In particular, these exist for spin structures which do not admit any non-trivial harmonic spinor. Despite the fact that $\mc E$ does not enjoy any natural convexity property, we note that the existence of the negative gradient flow as shown in \cite{amwewi12} still holds in two dimensions. Second, if $K_g$ denotes the Gau{\ss} curvature of $g$, the Schr\"odinger-Lichnerowicz formula implies
\beq\label{Dirac_energy}
\mc E(g,\phi)=\tfrac{1}{2}\int_M|D_g\varphi|^2\,dv^g-\tfrac{1}{4}\int_MK_g\,dv^g,
\ee
where $D_g$ is the Dirac operator associated with the spinor bundle $\Sigma_gM$. Since the second term in \cref{Dirac_energy} is topological by Gau\ss{}-Bonnet, we obtain immediately the topological lower bound
$$
\inf\mc E\geq\pi|\gamma-1|.
$$
We will show in \cref{inf.E} that we actually have equality. For the infimum we find a trichotomy of well-known spinor field equations. Namely, if $P_g$ is the twistor operator associated with $\Sigma_gM$ (see \cref{twi.spi} for its definition), then $(g,\phi)$ attains the infimum if and only if
$$
\begin{array}{ll}
P_g\phi=0,&\quad\gamma=0\\
\nabla^g\phi=0,&\quad\gamma=1\\
D_g\phi=0,&\quad\gamma\geq2,
\end{array}
$$
which matches the usual trichotomy for Riemann surfaces of positive, vanishing and negative Euler characteristic (\cref{trichotomy}, \cref{min_sphere}). Of course, any parallel spinor $\phi$ is also harmonic, i.e.\ $D_g\phi=0$. On the other hand, harmonic spinors on $M_\gamma$ are related to minimal immersions of the universal cover $\tilde M_\gamma$ into $\R^3$ via the spinorial Weierstra{\ss} representation (see \cite{kusner.schmitt:p95,schmitt:diss} or alternatively~\cite{fr98} where a nice 
presentation closer to our article is given). 
As a result we will be able to construct a plenitude of examples for various spin structures (\cref{min_att_or_not}). In particular, with the notable exception of $\gamma=2$, {\em for any genus} there exist critical points which are in fact absolute minimisers. Finally, we completely classify the critical points on the sphere (\cref{min_sphere}) and the flat critical points on the torus (\cref{classification}). 

\bigskip

\paragraph{\bf General conventions.}
In this article, $M_\gamma$ will denote the up to diffeomorphism unique closed oriented surface of genus $\gamma$. Further, $g$ will always be a Riemannian metric. Rotation on each tangent space by $\pi/2$ in the counterclockwise direction induces a complex structure $J$ which in particular is a $g$-isometry. More concretely, a local positively oriented $g$-orthonormal basis $(e_1,e_2)$ satisfies $Je_1=e_2$ and $Je_2=-e_1$. Conversely, any complex structure determines a conformal class $[g]$ of Riemannian metrics. We will often tacitly identify $(e_1,e_2)$ with the dual basis $(e^1,e^2)$ via the musical isomorphisms $\sharp$ and $\flat$. The Riemannian volume form $\omega_g$ is then locally given by $e_1\wedge e_2$. Further, the dual complex structure $J^*$ acting on $1$-forms is simply $-\star$, where $\star$ is the usual Hodge operator sending $e_1$ to $e_2$ and $e_2$ to $-e_1$. The Levi-Civita connection associated with $g$ will be written as $\nabla^g$. The Gau{\ss} curvature $K_g$ is just half the scalar curvature $s_g$, i.e.\ $2K_g=s_g=-2R_g(e_1,e_2,e_1,e_2)$, where $R_g$ denotes the Riemannian $(4,0)$-curvature tensor defined by $R(e_1,e_2,e_1,e_2)=g([\nabla^g_{e_1},\nabla^g_{e_2}]e_1-\nabla^g_{[e_1,e_2]}e_1,e_2)$. In the sequel we shall often drop any reference to $g$ if the underlying metric is clear from the context. The {\em divergence} of a tensor $T$ is given by
\begin{equation}
\mr{div}_g T=-\sum_{k=1}^n(\nabla^g_{e_k}T)(e_k,\,\cdot\,).
\end{equation} 
Finally, we use the convention $v\odot w:=(v\otimes w+ w\otimes v)/2$ for the symmetrisation of a $(2,0)$-tensor. 
%

\section{Spin geometry}
%
\subsection{Spinors on surfaces}\label{prelim.spinor.surface}
We recall some spin geometric features of surfaces. Suitable general references are~\cite{fr00,lami89}. 

\medskip

Every oriented surface admits a spin structure, i.e.\ a twofold covering of $P_{\GL_+(2)}$, the bundle of positively oriented frames which restricted to a fibre induces the connected $2$-fold covering of $\GL_+(2)$. In particular, spin structures on $M_\gamma$ are classified by elements of $H^1(P_{\GL_+(2)},\Z_2)$ whose restriction to the fibre gives the non-trivial class. From the exact sequence associated with the fibration $\GL_+(2)\to P_{\GL_+(2)}\to M_\gamma$ it follows that spin structures are in $1-1$ correspondence with elements in $H^1(M_\gamma,\Z_2)$. Hence there exist $2^{2\gamma}=\# H^1(M_\gamma,\Z_2)$ isomorphism classes of spin structures on $M_\gamma$.

\medskip

A pair $(M_\gamma,\sigma)$ consisting of a genus $\gamma$ surface and a fixed spin structure $\sigma$ will be called a {\em spin surface}. If, in addition, we also fix a metric, we can consider $\Sigma_gM\to M$, the complex bundle of {\em Dirac spinors} associated with the complex unitary representation $(\Delta,h)$ of $\Spin(2)$. Note that the action of $\omega_g$ splits $\Delta$ into the irreducible $\mp i$ eigenspaces $\Delta_\pm\cong\C$. This gives rise to a global decomposition 
$$
\Sigma_g=\Sigma_{g+}\oplus\Sigma_{g-}
$$
into {\em positive} and {\em negative (Weyl) spinors}. Further, since $\Delta_-\cong\bar\Delta_+$, $\Delta\cong\C\oplus\bar\C$ carries a {\em quaternionic structure}. Equivalently, there exists a $\Spin(2)$-equivariant map $\alpha:\Delta\to\Delta$ which interchanges $\Delta_+$ and $\Delta_-$ and squares to minus the identity. Hence we can think of $\Delta$ as the quaternions $\mathbb H$ with real inner product $\langle\cdot\,,\cdot\rangle:=\mr{Re}\,h$. Locally, we can represent spinors in terms of a local orthonormal basis of the form $(\phi,e_1\cdot\phi,e_2\cdot\phi,\omega\cdot\phi)$, where $\phi$ is a unit spinor and $(e_1,e_2)$ a local positively oriented orthonormal basis. In particular,
\beq\label{nabdec}
\nabla_X \phi = A(X)\cdot \phi + \beta(X) \omega\cdot \phi
\ee
for a uniquely determined endomorphism field $A\in \Gamma(\End(TM))$ and a $1$-form $\beta\in\Omega^1(M)$. We also say that the pair $(A,\beta)$ is {\em associated} with $(g,\phi)$. Note that $A$ and $\beta$ determine the spinor field $\phi$ up to a global constant in the following sense. If $\phi_1$ and $\phi_2$ are unit spinor fields, and if they both solve~\cref{nabdec} for $\phi=\phi_i$, then there is a unit quaternion $c$ such that $\phi_1=\phi_2 c$. Hence an orbit of the action of the unit quaternions $\Sp(1)$ on unit spinor fields is determined by a pair $(A,\beta)$ for which a solution to~\cref{nabdec} exists. The question of determining the pairs which can actually arise will be addressed in~\cref{integrability}.

\medskip

As pointed out above, the choice of a Riemannian metric induces a complex and in fact a K\"ahler structure on $M_\gamma$. In particular, we can make use of the holomorphic picture of spinors on Riemann surfaces~\cite{at71,hi74}. Here, spin structures on $(M_\gamma,[g])$ are in 1--1 correspondence with holomorphic square roots $\lambda$ of the canonical line bundle $\kappa_\gamma=T^*\!M^{1,0}$, i.e.\ $\lambda\otimes\lambda\cong\kappa_\gamma$ as {\em holomorphic} line bundles. The corresponding spinor bundle is given by
\ben
\Sigma_g=\Lambda^*TM^{1,0}\otimes\lambda\cong\lambda\oplus\lambda^*
\ee
where we used the identification $TM^{1,0}\cong T^*\!M^{0,1}$ as {\em complex} line bundles. Clifford multiplication is then given by $v\cdot\varphi=\sqrt{2}(v\wedge\varphi-\iota(v^*)\varphi)$ where $v\in T^*\!M^{0,1}$ and $\iota(v^*)$ denotes contraction with the hermitian adjoint of $v$. The resulting even/odd-decomposition $\Sigma_g=\lambda\oplus\lambda^*$ is just the decomposition into positive and negative spinors.
%
\subsection{Dirac operators}~\label{prelim.spinor.surface.two}
Associated with any spin structure is the Dirac operator 
$$
D_g:\Gamma(\Sigma_gM)\to\Gamma(\Sigma_gM)
$$
which is locally given by $D_g\phi=e_1\cdot\nabla^g_{e_1}\phi+e_2\cdot\nabla^g_{e_2}\phi$. We have the useful formul{\ae}
$$
\omega\cdot D\phi=-D(\om\cdot\phi)\quad\mbox{and}\quad\<\omega\cdot D\phi, D\phi\>+ \< D\phi,\omega\cdot D\phi\>=0.
$$
In particular, for $a,b\in \mR$ with $a^2+b^2=1$ we obtain
\beq\label{circle.action}
|D(a\phi +b \omega\cdot \phi)|^2=a^2|D\phi|^2+b^2|\omega D\phi|^2=|D\phi|^2.
\ee
In terms of the pair $(A,\beta)$ determined by $\phi$ we have
\beq\label{diracpair}
D\phi=\sum_{k=1}^2e_k\cdot A(e_k)\cdot\phi+\beta(e_k)e_k\cdot\omega\cdot\phi=\tr A\,\phi +\tr(A\circ J)\omega\cdot\phi-(\beta\circ J)^\sharp\cdot\phi. 
\ee 
Moreover, restriction of $D_g$ to $\Sigma_{g\pm}$ gives rise to the operators $D_g^\pm:\Gamma(\Sigma_{g\pm})\to\Gamma(\Sigma_{g\mp})$. 

\medskip

A remarkable fact we shall use repeatedly is the conformal equivariance of $D$ in the following sense~\cite{hi74}. If for $u\in C^\infty(M)$ we consider the metric $\tilde g=e^{2u}g$ conformally equivalent to $g$, we have a natural bundle isometry $\Sigma_g\to\Sigma_{\tilde g}$ sending $\phi$ to $\tilde\phi$. Furthermore,
\begin{equation}\label{confdirac}
\tilde D\tilde\phi=e^{-3u/2}\widetilde{De^{u/2}\phi}
\end{equation}
where we let $\tilde D=D_{\tilde g}$. Note that for a vector field $X$ we have $\widetilde{X\cdot\phi}=\tilde X\cdot\tilde\phi$ if $\tilde X=e^{-u}X$~\cite[(1.15)]{bfgk91}. In particular, the dimension of the space of {\em harmonic spinors}, $\ker D$, as well as the spaces of  (complex) {\em positive} and {\em negative harmonic spinors}, $\ker D^+$ and $\ker D^+$, are conformal invariants. This is also manifest in terms of the holomorphic description above. Namely, after choosing a complex structure, i.e.\ a conformal class on $M_\gamma$, and a holomorphic square root $\lambda$ of $\kappa_\gamma$, we have
$$
D\phi=\sqrt{2}(\del_\lambda+\del^*_\lambda)\phi 
$$
where $\del_\lambda:\Gamma(\lambda)\to\Gamma(T^*\!M^{0,1}\otimes\lambda)$ is the induced Cauchy-Riemann operator on $\lambda$ whose formal adjoint is $\del^*_\lambda$. In particular, a positive Weyl spinor $\phi$ is harmonic if and only if the corresponding section of $\lambda$ is holomorphic. Note that $\mr{coker}\,D^+\cong\ker D^-$ so that $\dim\ker D^+=\dim\ker D^-$ by the Atiyah-Singer index theorem. An explicit isomorphism is provided by the quaternionic structure from \cref{prelim.spinor.surface} which maps positive harmonic spinors to negative ones and vice versa.
%
\subsection{Bounding and non-bounding spin structures}\label{bounding.nonbounding}
The orientation-preserving diffeomorphism group $\diff_+(M_\gamma)$ acts on the bundle of oriented frames and therefore permutes the possible spin structures on $M_\gamma$ by its action on $H^1(P_{\GL_+(2)},\Z_2)$ resp.\ $H^1(M_\gamma,\Z_2)$. There are precisely two orbits, namely the orbits of {\em bounding} and {\em non-bounding} spin structures. They contain $2^{\gamma-1}(2^\gamma+1)$ respectively $2^{\gamma-1}(2^\gamma-1)$ elements~\cite{at71}. In particular, on the $2$-torus where $\gamma=1$, there is a unique non-bounding spin structure and three bounding ones. These two orbits correspond to the two spin cobordisms classes of $M_\gamma$~\cite{mi65}. Recall that in general, a spin manifold $(M,\sigma)$ is {\em spin cobordant to zero} if there exists an orientation preserving diffeomorphism to the boundary of some compact manifold so that the naturally induced spin structure on the boundary (see for instance \cite[Proposition II.2.15]{lami89}) is identified with $\sigma$ under this diffeomorphism. Numerically, we can distinguish these two orbits as follows. Fix a complex structure on $M_\gamma$ and identify the set of spin structures with the holomorphic square roots $\mc S(M_\gamma)$ of the resulting canonical line bundle $\kappa_\gamma$. Let $d^+(g):=\dim_\C\ker D^+=\dim_\C H^0(M_\gamma,\lambda)$. Then 
$$
\varrho:\mc S(M_\gamma)\to\Z_2,\quad\varrho(\lambda)\equiv d^+(g)\mod2
$$
is a quadratic function whose associated bilinear form corresponds to the cup product on $H^1(M_\gamma,\Z_2)$. Moreover, $\varrho(\lambda)=0$ if and only if $\lambda$ corresponds to a bounding spin structure~\cite{at71}. For instance, it is well-known that on a torus, $d^+(g)$ is either $0$ or $1$~\cite{hi74}. Therefore, the three bounding spin structures do not admit positive harmonic spinors (regardless of the conformal structure), while the non-bounding one (the generator of the spin cobordism class) admits a harmonic spinor. As a further application, we note that $d(g)=\dim_\C\ker D=2d^+(g)$ is divisible by $4$ if and only if $\Sigma$ is a bounding spin structure.
%

\subsection{The spinorial Weierstra{\ss} representation}\label{spinor.weier}
Let $(M_\gamma,\sigma,g,H)$ be a spin surface with fixed Riemannian metric $g$ and $H\in C^\infty(M_\gamma)$. The universal covering of $M_\gamma$ will be denoted by $\widetilde M_\gamma$. 
Let $S_{\gamma,\sigma,g,H}$ be the set of solutions of 
\begin{equation}\label{spin.weier.cond}
D\phi=H\phi\qquad |\phi|\equiv 1
\end{equation} 
on $M_\gamma$. Finally, let $\Imm_{\gamma,g,H}$ be the set of all periodic isometric immersions $F:\widetilde M_\gamma\to \R^3$ with mean curvature $H:M_\gamma\to \R$. 
Here, periodic refers to the existence of a `period homomorphism' $P:\pi_1(M_\gamma)\to \R^3$
with ${F([\alpha]\cdot x)} = F(x) +P([\alpha])$ for all $x\in \witi M_\gamma$ and 
$[\alpha]\in \pi_1(M_\gamma)$ which we view as a Deck transformation of 
$\widetilde M_\gamma \to M_\gamma$. Note that $\R^3$ acts on $\Imm_{\gamma,g,H}/\R^3$ by translations. The spinorial Weierstra{\ss} representation is the statement that there is a natural covering map 
$$
W:S_{\gamma,\sigma,g,H} \to \Imm_{\gamma,g,H}/\R^3
$$
with $W(\phi)=W(\psi)$ if and only if $\phi=\pm\psi$. It follows from Smale-Hirsch-theory (which is a special case of the h-principle for immersions of surfaces in 3-dimensional manifolds) that the image of $W$ 
is exactly one connected component of $\Imm_{\gamma,g,H}/\R^3$. This establishes a bijection from the set of isomorphism classes of spin structures on $M_\gamma$ to the set of connected components of 
$\Imm_{\gamma,g,H}/\R^3$. 

We briefly sketch the constructions of $W$ and of its `inverse', for details  
we refer to \cite{baer:98} and \cite{fr98} which express 
these constructions in modern language.

The `inverse' of the map $W$ can be described as follows: We fix a parallel spinor $\Psi$ on $\R^3$ of constant length $1$.
If $\witi M_\ga\to \R^3$ is a periodic isometric immersion 
then the `restriction' $\psi:=\Psi|_{\witi M_\ga}$ is a periodic unit spinor on  $\witi M_\ga$, 
well-defined
up to a sign. One can show that $\psi$ is the pullback of a unit spinor on $M_\ga$ under the universal covering map, provided that $M_\ga$
is equipped with the right choice of spin structure. 

To describe the map $W$ itself, one shows that a unit spinor $\phi$ on $M_\ga$
yields a linear isometric embedding $A_x:T_xM_\ga\to \R^3$ for every $x\in M_\ga$. The equation \eqref{spin.weier.cond} 
is then equivalent to the fact that these maps $A_x$ integrate, i.e. that there is a map $F:\witi M_\ga\to \R^3$ with $d_xF=A_x$.
 
The Weierstra\ss{} representation has a long history.  For minimal surfaces it can be traced back 
to Weierstra\ss{}' work on conformal parametrisations. For arbitrary surfaces the history is less clear, as surface representations in terms 
of different data were proven. The earliest reference known to us which shows that surfaces can be represented by solutions of $D\phi=H\phi$
is the preprint \cite{kusner.schmitt:p95} based on N.~Schmitt's thesis \cite{schmitt:diss}.
Necessity of  \eqref{spin.weier.cond} was certainly known before, 
see for instance~\cite{trautman:92,trautman:95}, and related results were already obtained in \cite{kenmotsu:79} and reportedly 
by Eisenhart and Abresch.
%
%
%
\section{Critical points}
%
\subsection{The Euler-Lagrange equation}
First we express the negative gradient of $\mc E$ in~\cref{neggrad} in terms of $A$ and $\beta$ as defined by~\cref{nabdec}. We write $|A|$ for the induced $g$-norm of $A$, i.e.\ $|A|^2=\tr\,A^tA$. Further, for a symmetric 2-tensor $h$ we denote by $h_0 = h - \tfrac 12\tr \,h \cdot g$ its traceless part. 

\begin{proposition}\label{critical}
The negative gradient of $\mc E$ is given by
\begin{align*}
Q_1(g,\phi) &= {-\tfrac{1}{4}}(\nabla_{J(\,\cdot\,)} \beta)^{sym}+\tfrac 12(A^tA + \beta \otimes \beta)_0\\
Q_2(g,\phi) &=- (\divv A) \cdot \phi -(\divv \beta)\, \omega\cdot \phi.
\end{align*}
\end{proposition}
\begin{proof}
First, with $A(e_i)=\sum_kA_{ki}e_k$ for a $g$-orthonormal basis $(e_1,e_2)$,
\begin{align*}
\langle \nabla \phi \otimes \nabla \phi \rangle &= \sum_{i,j} \langle \nabla_{e_i} \phi, \nabla_{e_j} \phi \rangle e_i \otimes e_j\\
&= \sum_{i,j} \langle A(e_i)\cdot\phi + \beta(e_i)\omega\cdot\phi,A(e_j)\cdot\phi+\beta(e_j)\omega\cdot\phi\rangle e_i \otimes e_j\\
&= \sum_{i,j} \bigl(\langle A(e_i)\cdot\phi, A(e_j)\cdot\phi\rangle + \beta(e_i)\beta(e_j) \bigr) e_i \otimes e_j\\
&= \sum_{i,j} \bigl( \sum_k A_{ki}A_{kj} + \beta(e_i)\beta(e_j) \bigr) e_i \otimes e_j\\
&= A^tA + \beta \otimes \beta
\end{align*}
and
\ben
|\nabla \phi|^2 = \tr \langle \nabla \phi \otimes \nabla \phi \rangle = \tr (A^tA) + \tr (\beta\otimes \beta) = |A|^2+|\beta|^2. 
\ee
On the other hand, $\<X\wedge Y\cdot \phi, A(Z)\cdot \phi\>=0$ and $\<X\wedge Y\cdot\phi,\omega\cdot\phi\>=\omega(X,Y)$, using the convention $e_1 \wedge e_2 = e_1 \otimes e_2 - e_2 \otimes e_1$. This implies
\ben
T_{g,\phi}(X,Y,Z)=\tfrac{1}{2}\omega(X,Y)\beta(Z)+\tfrac{1}{2}\omega(X,Z)\beta(Y)
\ee
and therefore
\begin{align}
&\div T_{g,\phi} = -\tfrac{1}{2}\sum_{i,k,l} \bigl( \omega(e_i,e_k) (\nabla_{e_i} \beta) (e_l) + \omega(e_i,e_l) (\nabla_{e_i}\beta)(e_k) \bigr) e_k \otimes e_l\nonumber\\
=&(\nabla_{e_2}\beta)(e_1) e_1 \otimes e_1 - (\nabla_{e_1} \beta) (e_2)e_2 \otimes e_2+\bigl((\nabla_{e_2} \beta) (e_2) - (\nabla_{e_1} \beta) (e_1)\bigr) e_1 \odot e_2\nonumber\\
=& (\nabla_{J(\,\cdot\,)} \beta)^{sym}.\label{div.T.comp}
\end{align}
Next we work pointwise with a synchronous frame. Since vector fields anticommute with $\omega$,
\begin{align*}
\nabla^*\nabla \phi=& -\sum_{i=1}^2(\nabla_{e_i}\nabla_{e_i}\phi -\nabla_{\nabla_{e_i}e_i}\phi)\\
=&  -\sum_{i=1}^2 \Bigl(A(e_i)\cdot A(e_i)\cdot\phi +\beta(e_i)\bigl(A(e_i)\cdot \om+ \om \cdot A(e_i)\bigr)\cdot \phi+\beta(e_i)^2\om\cdot\om\cdot\phi\\
& \phantom{\sum}+ \nabla_{e_i}\bigl(A(e_i)\bigr)\cdot \phi  
+ \nabla_{e_i}\bigl(\beta(e_i)\bigr)\omega\cdot \phi\Bigr)\\
=& (|A|^2+|\beta|^2) \phi+(\divv A) \cdot \phi +(\divv \beta)\, \om \cdot \phi. 
\end{align*}
Since $Q_2(g,\phi)$ is orthogonal to $\phi$ we must have
\ben
Q_2(g,\phi)= - (\divv A) \cdot \phi -(\divv \beta)\, \om \cdot \phi,
\ee
whence the assertion.
\end{proof}

In terms of the pair $(A,\beta)$ we can now characterise a critical point as follows.

\begin{corollary}\label{beta-is-harmonic}
A pair $(g,\varphi)$ is a critical point of $\mc E$ if and only if
\beq\label{critpointchar}
\div\beta=0, \quad\div A=0 , \quad(\nabla_{J(\,\cdot\,)}\beta)^{sym}=2(A^tA+\beta\otimes\beta)_0.
\ee
In particular, if $(g,\varphi)$ is critical, then 
\begin{enumerate}[{\rm (i)}]
	\item\label{beta.closed} $\tr\,Q_1(g,\phi)=\star d\beta/4=0$, hence $\beta$ is a harmonic $1$-form.
	\item\label{nabla.sym} $\nabla_{J(\,\cdot\,)}\beta$ is traceless symmetric, i.e.\ $(\nabla_{J(\,\cdot\,)}\beta)_0=0$ and $(\nabla_{J(\,\cdot\,)}\beta)^{sym}=\nabla_{J(\,\cdot\,)}\beta$.
	\item\label{nabla.com.lin} $\nabla_{J(X)}\beta(Y)=\nabla_X\beta(J(Y))$.
	\item\label{div.beta.beta} $\div (\beta \otimes \beta)_0=0$
\end{enumerate}
\end{corollary}
\begin{proof}
\cref{critpointchar} follows directly from~\cref{critical}. For \eqref{beta.closed}, we note that
\beq\label{eq.tr.div.t}
\tr\div T_{g,\phi} = (\nabla_{e_2} \beta)(e_1) - (\nabla_{e_1} \beta)(e_2)=-\star d\beta,
\ee
whence $4\tr\,Q_1=\star d\beta$ from \cref{neggrad}. For \eqref{nabla.sym} and \eqref{nabla.com.lin} we note that in an orthonormal frame the anti-symmetric part of $\nabla_{J(\,\cdot\,)}\beta$ is given by
\ben
(\nabla_{J(e_2)} \beta) (e_1)- (\nabla_{J(e_1)} \beta) (e_2) = -(\nabla_{e_1} \beta) (e_1)- (\nabla_{e_2} \beta) (e_2)= \div \beta.
\ee
Hence $\nabla_{J(\,\cdot\,)}\beta$ is symmetric if and only if $\div \beta = 0$. Since $\nabla \beta$ is symmetric if and only if $d \beta=0$,
\ben
\nabla_{J(X)}\beta (Y) = \nabla_{J(Y)} \beta(X) = \nabla_X \beta(J(Y))
\ee
if $(g,\varphi)$ is critical. To prove \eqref{div.beta.beta} we observe $\tr\,\beta\otimes\beta=|\beta|^2$ so that $(\beta \otimes \beta)_0 = \beta \otimes \beta - \tfrac 12 |\beta|^2 g$. Now in a synchronous frame
\begin{align*}
\div \beta \otimes \beta &= - (\nabla_{e_1} \beta) (e_1) \beta - \beta(e_1) \nabla_{e_1} \beta - (\nabla_{e_2} \beta )(e_2) \beta - \beta(e_2) \nabla_{e_2} \beta\\ 
&= (\div \beta) \beta - \nabla_{\beta^\sharp}\beta,
\end{align*}
whence $\div \beta \otimes \beta = - \nabla_{\beta^\sharp}\beta$ if $\div \beta = 0$. Moreover,
\begin{align*}
\div |\beta|^2 g &= - d |\beta|^2
= -2 g(\nabla \beta, \beta)\\ 
&=-2 \sum_{i,j} (\nabla_{e_i} \beta) (e_j)\beta(e_j)e_i\\
&=-2 \sum_{i,j} ((\nabla_{e_j} \beta) (e_i)+d\beta(e_i,e_j))\beta(e_j)e_i\\
&= - 2 \nabla_{\beta^\sharp}\beta + 2 \iota_{\beta^\sharp} d\beta.
\end{align*}
Consequently, $\div |\beta|^2g =-2\nabla_{\beta^\sharp} \beta$ if $d\beta =0$, whence the assertion.
\end{proof}

\begin{remark}
\phantom{a}\hfill
\begin{enumerate}[(i)]
	\item The proof of properties \eqref{nabla.sym} to \eqref{div.beta.beta} solely uses the harmonicity of $\beta$.
	\item The identity~\eqref{circle.action} induces a circle action which preserves the functional $\mc E$. Together with the quaternionic action on $\Delta$ we see that there is a $\U(2) = S^1\times_{\Z_2} \SU(2)$-action which preserves the functional and therefore acts on the critical points (cf.\ also~\cite[Section 4.1.3, Table 2]{amwewi12}).
\end{enumerate}
\end{remark}

The condition that $Q_1(g,\phi)$ is trace-free or equivalently, that the associated $1$-form $\beta$ is closed, can be interpreted as follows. As pointed out in~\cref{prelim.spinor.surface}, there is a natural bundle isometry $\mc C:\Sigma_g\to\Sigma_{\tilde g}$ between conformally equivalent metrics $\tilde g=e^{2u}g$, $u\in C^\infty(M)$. Hence, for $(g,\phi)\in\mc N$ we can consider the associated {\em spinor conformal class} $[g,\phi]:=\{(\tilde g,\tilde\phi)\,|\,\tilde g=e^{2u}g,\,\tilde\phi=\mc C\phi\}$.

\begin{proposition}\label{conf.class.min}
The following statements are equivalent:
\begin{enumerate}[{\rm(i)}]
	\item\label{enum.min.cla}$(g,\phi)\in\mc N$ is an absolute minimiser in its spinor conformal class.
	\item\label{enum.bet}$d\beta=0$.
	\item\label{enum.tra}$\tr\,Q_1(g,\phi)=0$.
\end{enumerate}
Furthermore, in any spinor conformal class there exists an absolute minimiser which is unique up to homothety. In particular, any spinor conformal class contains a unique absolute minimiser of total volume one.
\end{proposition}
\begin{proof}
The equivalence between~\eqref{enum.bet} and~\eqref{enum.tra} is just \cref{critical}. For~\eqref{enum.bet} $\Rightarrow$~\eqref{enum.min.cla} assume that $\beta$ associated with $(g,\phi)$ satisfies $d\beta=0$. For any $(\tilde g,\tilde\phi)\in[g,\phi]$ we find
\ben
|\tilde D\tilde\phi|^2=e^{-3u}|De^{u/2}\phi|^2=e^{-2u}|D\phi+\tfrac{1}{2}\grad u\cdot\phi|^2
\ee
by~\cref{confdirac}. For all $u \in C^\infty(M)$ this and \cref{diracpair} gives
\begin{align}
\int_M |\tilde D\phi|^2 d \tilde v &= \int_M|D\phi|^2 + \tfrac 14 |du|^2 + \langle D \phi, \grad u \cdot \phi \rangle dv\nonumber\\
&=  \int_M|D\phi|^2 + \tfrac 14 |du|^2-\langle(\beta\circ J)^\sharp\cdot \phi, \grad u \cdot \phi \rangle dv\nonumber\\
&= \int_M|D\phi|^2 + \tfrac 14 |du|^2 + (\star\beta, du)dv\nonumber\\
&= \int_M|D\phi|^2 + \tfrac 14 |du|^2 + (\star d\beta, u)dv\nonumber\\
&= \int_M|D\phi|^2 + \tfrac 14 |du|^2dv\nonumber\\
& \geq\int_M |D\phi|^2dv.\label{conf.cla.est}
\end{align}
Further, this yields that $\int_M|du|^2/4+(\star d\beta,u)dv\geq0$ for an absolute minimiser. Taking $u=-\star d\beta$ shows that $\beta$ associated with an absolute minimiser must be closed, hence~\eqref{enum.min.cla} $\Rightarrow$~\eqref{enum.bet}. Finally, equality holds in~\eqref{conf.cla.est} if and only if $u$ is constant. To prove existence of an absolute minimiser we first note that for the $1$-form $\tilde\beta$ associated with $(\tilde g,\tilde\phi)\in[g,\phi]$ we have $\tilde\beta(\tilde X)=e^{-u}\tilde\beta(X)=\langle\tilde\nabla_{\tilde X}\tilde\phi,\tilde\omega\cdot\tilde\phi\rangle$. On the other hand,
\ben
\langle\tilde\nabla_{\tilde X}\tilde\phi,\tilde\omega\cdot\tilde\phi\rangle=e^{-u}\beta(X)+\tfrac{1}{2}\langle X\cdot\grad\,e^{-u}\cdot\phi,\omega\cdot\phi\rangle
\ee
by~\cite[(1.15)]{bfgk91}. The latter term equals $J(X)(e^{-u})/2=de^{-u}(J(X))/2$ which implies
\ben
\tilde\beta=\beta-\tfrac{1}{2}\star du.
\ee
If $\beta=H(\beta)\oplus d[\beta]\oplus\delta\{\beta\}$ is the Hodge decomposition of $\beta$ for a function $[\beta]$ and a $2$-form $\{\beta\}$, then $d\tilde\beta=d(\delta\{\beta\}-\tfrac{1}{2}\star du)$. Taking $u=-2\star\{\beta\}$ yields that $d\tilde\beta=0$. 
\end{proof}
%
\subsection{Curvature}
Next we investigate the relationship between $A$, $\beta$ and the Gau\ss{} curvature $K$ of $g$. The basic link between curvature, spinors and 1-forms are the formul{\ae} of Weitzenb\"ock type
\beq\label{weitzenboeck}
D^2\phi=\nabla^*\nabla\phi+\frac{1}{2}K\cdot\phi\quad\mbox{ and }\quad\Delta\beta=\nabla^*\nabla\beta+K\cdot\beta.
\ee 
In particular, if $(g,\phi)$ is a critical and $g$ is flat, $\beta$ is necessarily parallel. We shall need a technical lemma first.

\begin{lemma}\label{Tr_div_T}
Let $\Phi=(g,\phi) \in \mc N$. Then $\langle D^2 \phi, \phi \rangle = |D\phi|^2-\star d\beta$.
\end{lemma}
\begin{proof}
A pointwise computation with a synchronous frame implies
\begin{align*}
\tr\div T_{g,\phi}  =&  -\sum_{j,k=1}^n(\nabla_{e_j}T_\phi)(e_j,e_k,e_k)\\
=&  -\sum_{k,j=1}^ne_j\<e_j\cdot e_k\cdot\phi,\nabla_{e_k}\phi\>-\sum_{k=1}^ne_k.\<\varphi,\nabla_{e_k}\varphi\>\\
=&  -\sum_{k,j=1}^n\<e_j\cdot e_k\cdot\nabla_{e_j}\phi,\nabla_{e_k}\phi\>-\sum_{k=1}^n\<e_j\cdot e_k\cdot\varphi,\nabla_{e_j}\nabla_{e_k}\varphi\>\\
&-|\nabla\varphi|^2 +\<\varphi,\nabla^\ast\nabla\varphi\>\\
=&  \<D^2\phi,\phi\>-|D\phi|^2.
\end{align*}
On the other hand, as already observed in \cref{eq.tr.div.t}, $\tr \div T_{g,\phi} =- \star d\beta$, whence the result in view of \cref{critical}.
\end{proof}

In terms of the associated pair $(A,\beta)$, the equations in~\eqref{weitzenboeck} read as follows.

\begin{proposition}\label{curvature} 
Let $(g,\phi)\in\mc N$. Then
\begin{enumerate}[{\rm (i)}]
	\item\label{K.A.beta} $K= 4\det A - 2\star d\beta$
	\item\label{K.div.beta} $K\star\beta=\div \nabla_{J(\,\cdot\,)}\beta$.
\end{enumerate}
\end{proposition}
\begin{proof}
\eqref{K.A.beta} Since we always have $\langle \nabla^*\nabla \phi, \phi \rangle = |\nabla \phi|^2$ for a unit spinor we get  
\ben
\tfrac{K}{2} = |D \phi|^2 - |\nabla \phi|^2-\star d\beta
\ee
from \cref{Tr_div_T} and the Schr\"odinger-Lichnerowicz formula. Locally,
\begin{align*}
|D \phi|^2 &= | \sum_i e_i\cdot\nabla_{e_i}\phi|^2
= \sum_{i,j} \langle e_i\cdot\nabla_{e_i}\phi, e_j\cdot\nabla_{e_j}\phi \rangle\\
&= |\nabla\phi|^2 + \sum_{i\neq j} \langle e_i\cdot\nabla_{e_i}\phi, e_j\cdot\nabla_{e_j}\phi \rangle
\end{align*}
and therefore
\begin{align*}
K + 2 \star d\beta&= 4 \langle e_1\cdot\nabla_{e_1} \phi, e_2\cdot\nabla_{e_2} \phi \rangle\\
&= 4 \langle e_1\cdot A(e_1)\cdot\phi + e_1\cdot \beta(e_1)\omega\cdot\phi,e_2\cdot A(e_2)\cdot\phi + e_2\cdot \beta(e_2)\omega\cdot\phi\rangle\\
&= 4\langle e_1\cdot A(e_1)\cdot\phi - \beta(e_1)e_2\cdot\phi,e_2\cdot A(e_2)\cdot\phi + \beta(e_2)e_1\cdot\phi\rangle\\
&= 4 \langle e_1\cdot A(e_1)\cdot\phi, e_2\cdot A(e_2)\cdot\phi\rangle\\ 
&= 4 \langle -A_{11}\phi +A_{21}e_1\cdot e_2\cdot\phi, -A_{12}e_1\cdot e_2\cdot\phi-A_{22}\phi\rangle\\
&= 4(A_{11}A_{22}-A_{21}A_{12}) = 4 \det A,
\end{align*}
where $(A_{ij})$ is the matrix of $A$ with respect to the basis $\{e_1,e_2\}$.

\smallskip

\eqref{K.div.beta} Computing in a synchronous frame yields
\begin{align*}
\div \nabla_{J(\,\cdot\,)} \beta  & = - \nabla_{e_1}\nabla_{J(e_1)} \beta - \nabla_{e_2}\nabla_{J(e_2)} \beta \\ 
&= - \nabla_{e_1}\nabla_{e_2} \beta + \nabla_{e_2}\nabla_{e_1} \beta= -R(e_1,e_2) \beta.
\end{align*}
Since $R(e_1,e_2)\beta=-K\star \beta$, \eqref{K.div.beta} follows.
\end{proof}

\begin{corollary}\label{cor_gauss}
If $(g,\phi)\in\mc N$ is a critical point of $\mc E$, then 
\begin{enumerate}[{\rm (i)}]
	\item\label{K.detA} $K= 4\det A$.
	\item\label{K.div.AA} $K\star\beta=2\div(A^tA)_0$.
	\item\label{nabla.phi.K} $2|\nabla\phi|^2\geq|K|$.
\end{enumerate}
\end{corollary}
\begin{proof}
The first two statements are immediate consequences of~\cref{beta-is-harmonic} and the previous proposition. Further, by \cref{Tr_div_T}, the second line of \cref{neggrad}, and the assumption $0\leq |D\phi|^2= |\na \phi|^2+K/2$,
while
\beq\label{dirnab}
|D\phi|^2 =|\sum_{i=1}^2 e_i\cdot\nabla_{e_i} \phi|^2\leq\big(\sum_{i=1}^2 1 \cdot |\nabla_{e_i}\phi|\big)^2\leq 2 |\nabla\phi|^2,
\ee
whence \eqref{nabla.phi.K}.
\end{proof}
%
\subsection{Integrability of $(A,\beta)$}\label{integrability}
Next we address the question for which pairs $(A,\beta)$ a solution to~\cref{nabdec} exists. Towards that end we introduce the Clifford algebra valued $1$-form $\Gamma(X):=A(X)+\beta(X)\omega$ and define the connection
\ben
\widetilde\nabla_X\phi:=\nabla_X \phi - A(X)\cdot \phi - \beta(X) \omega\cdot \phi=\nabla_X\phi-\Gamma(X)\cdot\phi.
\ee
A solution to~\cref{nabdec} exists if and only if if we have a non-trivial $\widetilde\nabla$-parallel spinor field. In fact this is equivalent to the triviality of the spinor bundle in the sense of flat bundles for we may regard $\Sigma M$ as a ``quaternionic'' line bundle. This in turn is equivalent to the vanishing of the curvature $R^{\widetilde\nabla}$ and the triviality of the associated holonomy map $\pi_1(M,p)\to \Aut(\Sigma_pM)$. We have
\begin{align*}
&R^{\widetilde \nabla}(X,Y)\phi =(\widetilde\nabla_X\widetilde\nabla_Y-\widetilde\nabla_Y\widetilde\nabla_X -\widetilde\nabla_{[X,Y]})\phi\\
=&\widetilde\nabla_X(\nabla_Y \phi - \Gamma(Y)\cdot \phi)-\widetilde\nabla_Y(\nabla_X \phi- \Gamma(X)\cdot \phi ) -\nabla_{[X,Y]} \phi + \Gamma([X,Y])\cdot \phi\\ 
=&R^\nabla(X,Y)\phi - \nabla_X( \Gamma(Y)\cdot \phi ) + \nabla_Y( \Gamma(X)\cdot \phi) 
-\Gamma(X) (\nabla_Y \phi - \Gamma(Y)\cdot \phi)\\
&+\Gamma(Y)(\nabla_X \phi - \Gamma(X)\cdot \phi ) +\Gamma(\nabla_XY)\cdot \phi - \Gamma(\nabla_YX)\cdot\phi\\ 
=& R^\nabla(X,Y)\phi - (\nabla_X \Gamma)(Y)\cdot \phi  + (\nabla_Y \Gamma)(X)\cdot \phi +\Gamma(X)\Gamma(Y)\cdot \phi- \Gamma(Y)\Gamma(X)\cdot \phi \\  
=&R^\nabla(X,Y)\phi - d\Gamma(X,Y)\cdot \phi + [\Gamma(X),\Gamma(Y)]\phi,
\end{align*}
where $d\Gamma$ denotes the skew-symmetric part of the covariant derivative $\nabla\Gamma$, i.e.
\beq
d\Gamma(X,Y):=(\nabla_X\Gamma)(Y) - (\nabla_Y\Gamma)(X) =(\nabla_XA)(Y) - (\nabla_YA)(X)+d\beta(X,Y)\omega.
\ee 
Similarly, we define $dA(X,Y):=(\nabla_XA)(Y) - (\nabla_YA)(X)$. Now for an oriented orthonormal basis $(e_1,e_2)$ we find
\begin{align*}
[\Gamma(e_1),\Gamma(e_2)]=&[A(e_1),A(e_2)] + 2 \beta(e_2)A(e_1)\omega -2 \beta(e_1)A(e_2)\omega\\
=&2(\det A)\omega-2\beta(e_2)J(A(e_1))+2\beta(e_1)J(A(e_2)). 
\end{align*}
Since $2R^{\nabla}(e_1,e_2)\phi=K\omega\cdot\phi$ we finally get
\begin{align*}
R^{\widetilde\nabla}(e_1,e_2)\phi 
=&-\frac12 K \omega\cdot\phi -dA(e_1,e_2)\phi -d\beta(e_1,e_2)\omega\cdot\phi\\
&+2(\det A)\omega\cdot\phi - 2 \beta(e_2)J(A(e_1))\phi + 2 \beta(e_1)J(A(e_2))\phi. 
\end{align*}
Since $K=4\det A-2\star d\beta$ by~\cref{curvature}, this vanishes for all $\phi$ if and only if $dA(e_1,e_2) = - 2 \beta(e_2)J(A(e_1)) + 2 \beta(e_1)J(A(e_2))$. Since $M$ is K\"ahler, $\nabla J=0$, hence $\nabla_X(A\circ J)(Y)=(\nabla_XA)(JY)$. Writing the previous expression invariantly yields the following

\begin{proposition}\label{int.cond.Abeta}
If the pair $(A,\beta)$ arises from a spinor field as in~\eqref{nabdec}, then
\ben
\div(A\circ J)=-2(J \circ A \circ J)(\beta^\sharp).
\ee
\end{proposition}

Conversely, if the integrability condition of \cref{int.cond.Abeta} is satisfied, then there exists a {\em local} solution $\varphi$ to \cref{nabdec}. Moreover, $\varphi$ is uniquely determined up to multiplication by a unit quaternion from the right.
%
\subsection{Absolute minimisers}
In dimension $n\geq3$ the only critical points of the spinorial energy functional $\mc E$ are absolute minimisers with $\mc E(g,\phi)=0$ \cite{amwewi12}. This stands in sharp contrast to the surface case.

\begin{theorem}\label{inf.E}
On a spin surface $(M_\gamma,\sigma)$ we have
\ben
\inf\mc E=\pi|\gamma-1|.
\ee
\end{theorem}
\begin{proof}
The lower bound $\inf\mc E\geq\pi|\gamma-1|$ follows directly from the Schr\"odinger-Lichnerowicz and Gau{\ss}-Bonnet formul{\ae}, for 
\beq\label{low.est.E}
\tfrac{1}{2}\int_{M_\gamma} |\nabla \phi|^2\geq-\tfrac{1}{4}\int_{M_\gamma}K=\pi(\gamma-1)
\ee
which gives the estimate for $\gamma\geq1$. For the sphere, we use (iii) of \cref{cor_gauss} to obtain
\beq\label{dirnab2}
2\pi=\frac{1}{2}\int_{S^2}K\leq\int_{S^2} |\nabla\phi|^2.
\ee
Further, the results of \cref{sphere} show that this lower bound is actually attained on the sphere. For genus $\gamma\geq1$ we show the existence of ``almost-minimisers'', i.e.\ for every $\varepsilon>0$ there is a unit spinor $(g,\phi)$ such that $\mc E(g,\phi) \leq \pi|\ga-1|+\varepsilon$. There is a standard strategy for their construction by gluing together 2-tori with small Willmore energy in a flat $3$-torus $(T^3,g_0)$ and restricting the parallel spinors of $T^3$ to the resulting surface, see also~\cite{gigr10} and~\cite{tr08} (which we discuss further in \cref{min_surf}) for related constructions. 

To start with we define the Willmore energy of a piecewise smoothly embedded surface $F:M\to T^3$ by
$$
\mc W(F):=\frac12 \int_{F(M)}H^2dv^g.
$$
Here, $H$ is the mean curvature of $F(M)$ in $(T^3,g_0)$ and integration is performed with respect to the 
volume element $dv^g$ associated to the restriction of the Euclidean metric to $F(M)$. For sake of concreteness, consider a square fundamental domain of the torus in $\R^3$, fix $\rho>0$ and consider two flat disks of radius $\rho$ inside that domain which are parallel to the $(x_1,x_3)$-plane and are at small distance from each other. We want two replace the disjoint union of the disks of radius $\rho/2$ by a catenoidal neck and retain the vertical annular pieces. The result of this process will be called a {\em handle of radius $\rho$}. 

\begin{lemma}\label{lemma.gluing}
For all $\varepsilon >0$ there exists a handle of radius $\rho$ which has Willmore energy less than $\varepsilon$.
\end{lemma}
\begin{lemproof}
Since the Willmore energy is scaling invariant it suffices to construct a model handle with Willmore energy less than $\varepsilon$ for some radius $\rho(\varepsilon) > 0$. The solution for the given radius $\rho$ is then simply obtained by rescaling. We construct a model handle as a surface of revolution. It will be composed of a catenoidal part, a spherical part and a flat annular part. More precisely, let $L>0$ and consider the curve $\gamma=(\gamma_1,\gamma_2) : [0,\infty) \to \R \times (0,\infty)$ defined by
$$
\gamma(u) = \begin{cases}
 (\arsinh(u), \sqrt{1+u^2}) &,\,0 \leq u \leq L \\
 (a,b) + R\big(\cos(\frac{u-L}{R}-\alpha), \sin (\frac{u-L}{R}-\alpha)\big) &,\,L \leq u \leq L + \alpha R \\
 \big(a+R,b+u - (L+\alpha R)\big) &,\, L+\alpha R \leq u < \infty
 \end{cases}
$$
where we have set $(a,b) = (\arsinh(L)-L\sqrt{1+L^2}, 2 \sqrt{1+L^2})$, $R=1+L^2$ and $\alpha = \arcsin(1/\sqrt{1+L^2})$. Consider the surface of revolution around the $x_1$-axis defined by
$$
F(u,v)=\big(\gamma_1(u),\cos(v)\gamma_2(u), \sin(v)\gamma_2(u)\big)
$$
where $u\in[0,\infty)$, $v\in[0,2\pi)$. This surface is a piecewise smooth $C^1$-surface with Willmore energy
$$
\mathcal W(F)=\frac{\pi}{\sqrt{1+L^2}}
$$
which is precisely the Willmore energy of the spherical piece, the catenoid and the flat piece being minimal. We double this surface along the boundary $\{x_1=0\}$ and intersect with the region $\{x_2^2+x_3^2\leq 4b^2\}$ to get a handle of radius $\rho(L)=2b$ with Willmore energy $2\pi/\sqrt{1+L^2}<\varepsilon$ for $L$ big enough. This piecewise smooth handle may be approximated by smooth handles with respect to the $W^{2,2}$-topology to yield the desired smooth handle.
\end{lemproof}

\begin{remark}
Fix $\rho >0$ and consider the handle of radius $\rho$ with Willmore energy $\varepsilon = 4 \pi/\sqrt{1+L^2}$ which we obtain by rescaling the handle constructed above by $2b$. Then the distance between the flat annular pieces is given by
$$
2\frac{a+R}{2b}=\frac{1}{2}\left(\frac{\arsinh(L)}{\sqrt{1+L^2}}+\sqrt{1+L^2}-L\right)
$$
which goes to zero as $\varepsilon\to 0$ (i.e.\ $L\to\infty$).
\end{remark}

\begin{lemma}\label{lemma.inf-construction}
For a compact connected surface $M_\gamma$ of genus $\gamma\geq 1$ with a fixed spin structure $\sigma$, there is a flat torus $(T^3,g_0)$ and an embedding $F:M_\gamma\to T^3$ such that $\mc W(F)\leq \varepsilon$ and such that the spin structure on  $M_\gamma$ induced by this embedding is the given spin structure $\sigma$.
\end{lemma}

\begin{figure}
\begin{tikzpicture}[scale=.65]
\draw[thick] (0,0) --(0,8) --(8,8) --(8,0) --(0,0); 
\draw [blue, thick] (4.7,6) arc(0:180:1);
\draw [blue, thick] (4.7,2) arc(0:-180:1);
\draw [blue, thick] (2.7,6) -- (2.7, 2);
\draw [blue, thick] (4.7,6) -- (4.7, 2);
\draw [green, thick] (2.3,0) -- (2.3,8);
\draw [red, thick] (2.7,4.8) arc (0:180:.2);
\draw [red, thick] (2.7,5.5) arc (0:-180:.2);
\draw [red, thick] (2.7,2.5) arc (0:180:.2);
\draw [red, thick] (2.7,3.2) arc (0:-180:.2);
\draw[thick] (10,0) --(10,8) --(18,8) --(18,0) --(10,0); 
\draw [green, thick] (12.7,0) -- (12.7,8);
\draw [green, thick] (12.3,0) -- (12.3,8);
\draw [red, thick] (12.7,6) arc (0:180:.2);
\draw [red, thick] (12.7,6.7) arc (0:-180:.2);
\draw [red, thick] (12.7,4) arc (0:180:.2);
\draw [red, thick] (12.7,4.7) arc (0:-180:.2);
\draw [red, thick] (12.7,2) arc (0:180:.2);
\draw [red, thick] (12.7,2.7) arc (0:-180:.2);
\end{tikzpicture}
\caption{Surfaces with almost minimisers. The left-hand picture shows a torus with a non-bounding spin structure, drawn in green, and a torus with a bounding spin structure, drawn in blue. These surfaces are connected by necks drawn in red. The right-hand picture shows two tori with a non-bounding spin structure, drawn in green, connected by necks drawn in red.}\label{inf-construction}
\end{figure}
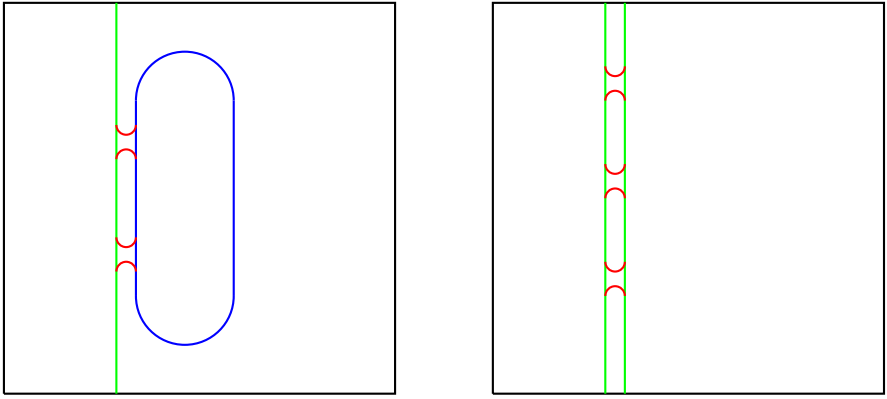

\begin{lemproof}
Since orientation preserving diffeomorphisms act transitively on both bounding and non-bounding spin structures, it is enough to show the lemma for only one bounding or non-bounding spin structure.

We deal with the case $\gamma=1$ first. For the non-bounding spin structure we may simply take $T_n$ to be any totally geodesic 2-torus in a flat torus $(T^3,g_0)$. This embedding has zero Willmore energy and the induced spin structure on $T_n$ is the non-bounding one. For a bounding spin structure we choose an embedding $D^2\subset T^2$ and let $S^1=\partial D^2$. Let $S^1_\delta$ denote the circle of length $\delta>0$ and set $T_b:=S^1 \times S^1_\delta \subset T^2 \times S^1_\delta$. Then $T_b$ has arbitrarily small Willmore energy for $\delta$ small enough, and the induced spin structure on $T_b$ is bounding. Note that we may slightly flatten the circle $S^1\subset T^2$ in order to make it contain a line segment. Then $T_b$ contains a flat disk which will be useful later for gluing in a handle.

In the higher genus case we use the tori $T_b$ and $T_n$ constructed above as building blocks which we connect by handles with small Willmore energy. The construction is illustrated in \cref{inf-construction}. If $\sigma$ is a non-bounding spin structure, we align a copy of $T_n$ and a copy of $T_b$ in such a way that $T_n$ is parallel and at small distance to a flat disk inside $T_b$. Then we connect $T_n$ and $T_b$ by $\gamma-1$ handles. If $\sigma$ is a bounding spin structure, we take two parallel copies of $T_n$ at small distance, and call them $T_n'$ and $T_n''$. Then we connect $T_n'$ and $T_n''$ by $\gamma-1$ handles. According to \cref{lemma.gluing} this can be done without introducing more than an arbitrarily small amount of Willmore energy. The resulting surface has genus $\gamma$ and carries a non-bounding spin-structure in the first, and a bounding spin structure in the second case.
\end{lemproof}

We return to the proof of \cref{inf.E}. With the notations of the lemma and the proposition we set $g:=F^*(g_0)$. Further, we restrict a parallel spinor of unit length on $T^3$ to $F(M_\gamma)$ and pull it back to a spinor $\phi$ on $(M_\gamma,\sigma,g)$. As in \cite{fr98} it follows that $D\phi = H\phi$, whence
$$
\frac12 \int_{M_\gamma}|D\phi|^2dv^g=\frac12 \int_{M_\gamma}H^2dv^g=\mc W(F)\leq\epsilon
$$
and thus 
$$
\mc E(g,\phi)=\frac12 \int_{M_\gamma}|D\phi|^2- \frac14 \int_{M_\gamma}K\leq\ep-\frac\pi2\chi(M)=\varepsilon +\pi|\gamma-1|
$$
as claimed.
\end{proof}

From \eqref{low.est.E}, the Schr\"odinger-Lichnerowicz formula and the results from \cref{spinor.weier} we immediately deduce the

\begin{corollary}
If $\gamma\geq1$, then $\mc E(g,\phi)=\pi|\gamma-1|$ if and only if $D_g\varphi=0$, that is, $\phi$ is a harmonic spinor of unit length. In particular, absolute minimisers of $\mc E$ over $M_\gamma$ correspond to  minimal isometric immersions of the universal covering of $M_\gamma$.
\end{corollary}

\begin{remark}
In the case of the sphere ($\gamma=0$) equality holds if and only if $\varphi$ is a so-called {\em twistor spinor}, see \cref{sphere}. Furthermore, as a consequence of the Schr\"odinger-Lichnerowicz and Gau\ss{}-Bonnet formula, a unit spinor on the torus ($\gamma=1$) is harmonic if and only if it is parallel.
\end{remark}

\begin{example}\label{min_surf}
For any $\gamma\geq 3$ there exists a triply periodic orientable minimal surface $M$ in $\R^3$ such that if $\Gamma$ denotes the lattice generated by its three periods, the projection of $M$ to the flat torus $T^3=\R^3/\Gamma$ is $M_\gamma$ \cite[Theorem 1]{tr08}. Since the normal bundle of $M_\gamma$ in $T^3$ is trivial there exists a natural induced spin structure which we claim to be a bounding one. To see this we need to analyse the construction in~\cite{tr08} which is a refinement of the construction used in \cref{lemma.inf-construction}. In a first step one starts with two flat minimal $2$-dimensional tori $T_1$ and $T_2$ inside the flat $3$-dimensional torus $T^3$. One can assume that $T_1$ and $T_2$ are parallel. The trivial spin structure on $T^3$ admits parallel spinors which we can restrict to parallel spinors on $T_1$ and $T_2$. In particular, both $T_1$ and $T_2$ carry the non-bounding spin structure so that the disjoint union $T_1\amalg T_2$ carries a bounding spin structure. Namely, $T_1\amalg T_2$ is the boundary of any connected component of 
$T^3\setminus (T_1\amalg T_2)$, and this even holds in the sense of spin manifolds, cf.\ also the discussion in \cite[Remark II.2.17]{lami89}. In a second step, small catenoidal necks are glued in between $T_1$ and $T_2$ but this does not affect the nature of the spin structure which thus remains a bounding one.
\end{example}

Using the conformal equivariance \eqref{confdirac} of the Dirac operator gives a further corollary. Namely, $\mc E(g,\phi)= \pi(\ga -1)$ for $(g,\phi)\in\mc N$ if and only if there is metric $\tilde g$ with nowhere vanishing spinor $\tilde\phi$ with $D_{\tilde g}\tilde\phi=0$. Indeed, for $g=|\tilde\phi|^4_{\tilde g}\tilde g$ the rescaled spinor $\phi=\tilde\phi/|\tilde\phi|_g$ is in the kernel of $D_g$ and of unit norm. 

\begin{corollary}
For $\gamma\geq1$ absolute minimisers on a spin surface correspond to nowhere vanishing harmonic spinors on Riemann surfaces.
\end{corollary}

For a generic conformal class on $M_\gamma$ nowhere vanishing harmonic spinors do \emph{not} exist. More exactly, it was proven in 
\cite{ammann.dahl.humbert:09} that the set $\mathcal{M}$ of all metrics $g$ with 
$\dim_\C\ker D_g\leq 2$ is open in the $C^1$-topology and dense in the $C^\infty$-topology in the set of all metrics. Thus, with similar arguments as in 
Lemma~\ref{dim_harm_spin} below it follows that nowhere vanishing harmonic spinors cannot exist in conformal classes $[g]$ with $g\in \mathcal{M}$. However, examples do exist for special cases such as hyperelliptic Riemann surfaces, where complex techniques can be used when regarding harmonic spinors as holomorphis sections as explained in~\cref{prelim.spinor.surface.two}. For instance, B\"ar
and Schmutz \cite{baer.schmutz:92} could compute the dimension of the space of harmonic spinors for any spin structure based on earlier work by Martens~\cite{martens:68} and Hitchin~\cite{hi74}. Complex geometry also gives us control on the zero set of the spinors, as highlighted by the following example.

\begin{example}\label{hol.sec}
Recall that hyperelliptic surfaces are precisely the Riemann surfaces of genus $\gamma\geq2$ which arise as two-sheeted branched coverings of the complex projective line (see for instance \cite[Paragraph \S7 and 10]{gu66}). There are exactly $2(\gamma-1)$ branch points $w_1,\ldots,w_{2(\gamma+1)}$, the so-called {\em Weierstra{\ss} points}. For any such Weierstra{\ss} point $w$, the divisor $2(\gamma-1)w$ defines the canonical line bundle $\kappa$ of $M_\gamma$, and $\lambda$ defined by $(\gamma-1)w$ is a holomorphic square root. In particular, there exists a holomorphic section $\varphi_0\in H^0(M_\gamma,\mc O(\lambda))$ -- a positive harmonic spinor -- whose divisor of zeroes is precisely $(\gamma-1)w$, that is, $\varphi_0$ has a unique zero of order $\gamma-1$ at $w$. Furthermore, on a hyperelliptic Riemann surface there exists a meromorphic function $f$ on $M$ with a pole of order $2$ at $w$ and a double zero elsewhere, say at $p\in M$. Hence, if the {\em genus of $M$ is odd}, then $\varphi_1=f^{(\gamma-1)/2}\varphi_0$ is a holomorphic section which has a unique zero at $p$. Regarding $\varphi_1$ as a negative harmonic spinor via the quaternionic structure therefore gives a non-vanishing harmonic spinor $\varphi_0\oplus\varphi_1\in\Gamma(\Sigma_g)$. Rescaling by its norm gives finally the desired absolute minimiser. Note that $\dim_\C H^0(M_\gamma,\mc O(\lambda))=(\gamma+1)/2$ (see for instance \cite[Theorem 14]{gu66}) so that $\lambda$ corresponds to a non-bounding spin structure if $\gamma\equiv1\mod4$, and to a bounding spin structure if $\gamma\equiv3\mod 4$. 
\end{example}

As already mentioned above, there are obstructions against absolut minimisers.

\begin{lemma}\label{dim_harm_spin}
If $(g,\phi)$ is an absolute minimiser over $M_\gamma$ with $\gamma\geq2$, then $d(g)=\dim_\C\ker D_g\geq4$.  
\end{lemma}
\begin{proof}
As noted in \cref{bounding.nonbounding}, $d(g)$ is even, so it remains to rule out the case $d(g)=2$. Viewing $\Sigma M\to M$ as a quaternionic line bundle with scalar multiplication from the right, $\ker D_g$ inherits a natural quaternionic vector space structure. In particular, it is a $1$-dimensional quaternionic subspace if $d(g)=2$. Since $D(1+i\omega)\phi=D\phi-i\omega D\phi=0$ there is a quaternion $q$ with $(1+i\omega)\phi= \phi q$. If $q\neq 0$, then $(1+i\omega)\phi$ is a nowhere vanishing section of the complex line bundle $\Sigma_+$ and thus yields a holomorphic trivialisation of the holomorphic tangent bundle via the holomorphic description of harmonic spinors in \cref{prelim.spinor.surface}. In particular, $\gamma=1$. If $q=0$, then $\phi$ is a nowhere vanishing section of $\Sigma_- \cong\overline\Sigma_+$ and a similar argument applies.
\end{proof}

Summarising, we obtain the following theorem concerning existence respectively non-existence of absolute minimisers.

\begin{theorem}\label{min_att_or_not}
On $(M_\gamma,\sigma)$ the infimum of $\mc E$ 
\begin{enumerate}[{\rm (i)}]
\item is attained in the cases
\begin{enumerate}[{\rm(a)}]
\item $\gamma=1$ and $\sigma$ is the non-bounding spin structure.
\item $\gamma\geq 3$ and $\sigma$ is a bounding spin structure.
\item $\gamma\geq5$ with $\gamma\equiv1\mod4$ and $\sigma$ is a non-bounding spin structure.
\end{enumerate}

\item is not attained in the cases
\begin{enumerate}[{\rm(a)}]
	\item $\gamma=1$ and $\sigma$ is a bounding spin structure.
	\item $\gamma=2$
	\item $\gamma=3,\,4$ and $\sigma$ is a non-bounding spin structure.
\end{enumerate}
\end{enumerate}
\end{theorem}

\begin{remark}
\phantom{a}\hfill
\begin{enumerate}[(i)]
	\item It remains unclear whether the infimum is attained for a non-bounding spin structure on surfaces of genus $\gamma\geq6$ and 
$\gamma\not\equiv1\mod4$. 
	\item In the case of the sphere ($\gamma=0$) the infimum of $\mc E$ is always attained. This will be discussed in \cref{sphere}.
\end{enumerate}
\end{remark}

\begin{proof}[Proof of Theorem \ref{min_att_or_not}]
(i) The non-bounding spin structure on $T^2$ is the one which admits parallel spinors, while (b) and (c) follow from \cref{min_surf} and \cref{hol.sec} respectively.

\smallskip

(ii) From \cref{bounding.nonbounding} we know that $d(g)$ must be divisible by $4$ if $\sigma$ is bounding while from Hitchin's bound $d(g)\leq\gamma+1$~\cite{hi74}. Therefore, under the conditions stated in (a) or (b), $d(g)\leq 3$ for any metric $g$ on $M_\gamma$ so that for a bounding $\sigma$ we necessarily have $d(g)=0$. If $\gamma\geq2$ we have $d(g)\geq4$ by \cref{bounding.nonbounding} and moreover, $d(g)\equiv2\mod4$ if $\sigma$ is non-bounding. Hence $d(g)\geq6$ which is impossible if $\gamma\leq4$.
\end{proof}

Finally, we characterise the absolute minimisers in terms of $A$ and $\beta$. First we note that $J$ induces a natural complex structure on $T^*\!M\otimes TM$ defined by
$$
i(\alpha\otimes v)=i\alpha\otimes v=\alpha\otimes iv:=\alpha\otimes Jv.
$$
Equipped with this complex structure, $T^*\!M\otimes TM$ becomes a complex rank $2$ bundle, and we have the complex linear bundle isomorphism
\begin{equation}\label{com.lin.iso}
T^*\!M\otimes TM\cong TM^{1,0}\otimes_\C(T^*\!M\otimes\C),\quad\alpha\otimes v\mapsto \alpha\otimes\tfrac 12(v-iJv).
\end{equation}
In this way, considering $A$ as a $TM$-valued $1$-form, the decomposition $\Omega^1(TM)\cong\Omega^{1,0}(TM^{1,0})\oplus\Omega^{0,1}(TM^{1,0})$ gives a decomposition
\ben
A = A^{1,0}+A^{0,1}.
\ee
Since $T^*\!M^{1,0}\otimes_\C TM^{1,0}$ is trivial we may identify $A^{1,0}$ with a smooth function $f:M\to\C$. Further, on any K\"ahler manifold $TM^{0,1}\cong T^*\!M^{1,0}$ so we may identify $A^{0,1}$ with a quadratic differential $q\in\Gamma(\kappa_\gamma^2)$. Finally, $\bar\partial f\in\Omega^{0,1}(M_\gamma)\cong \Gamma(TM^{1,0}_\gamma)$ and $\bar\partial q\in\Omega^{0,1}(\kappa_\gamma^2)\cong\Gamma (TM^{0,1}_\gamma)$. 

\begin{lemma}\label{div.free}
Modulo these isomorphisms we have
\ben
-\frac{1}{2}\div A^{1,0} = \bar\partial f \quad \text{and} \quad - \frac{1}{2}\div A^{0,1}= \overline{\bar\partial q}.
\ee
In particular, $\div A^{1,0}=0$ if and only if $\bar\partial f=0$ and $\div A^{0,1}=0$ if and only if $\bar\partial q =0$.
\end{lemma}
\begin{proof}
If we write
$$
A=\begin{pmatrix}a&c\\b&d\end{pmatrix}
$$
in terms of a positively oriented local orthonormal frame $\{e_i\}$, then
\beq\label{deco.endo}
A^{1,0}=\begin{pmatrix}\alpha&-\beta\\\beta&\alpha\end{pmatrix}=\tfrac 12\begin{pmatrix}a+d&-b+c\\b-c&a+d\end{pmatrix},
\quad
A^{0,1}=\begin{pmatrix}\gamma&\delta\\\delta&-\gamma\end{pmatrix}=\tfrac 12\begin{pmatrix}a-d&b+c\\b+c&-a+d\end{pmatrix}.
\ee
Hence $A^{1,0}$ is the sum of the trace and skew-symmetric part of $A$, while $A^{0,1}$ is the traceless symmetric part of $A$. Now fix a local holomorphic coordinate $z=x+iy$ and assume that $\{e_i\}$ is synchronous at $z=0$, i.e.\ $e_1(0)=\partial_x(0)$ and $e_2(0)=\partial_y(0)$. In particular, $\partial_z=(\partial_x-i\partial_y)/2$ corresponds to $e_1$ under the identification~\eqref{com.lin.iso}. From~\eqref{deco.endo}
\ben
A^{1,0}=(\alpha+i\beta)\,dz\otimes\partial_z
\quad\text{and}\quad
A^{0,1}=(\gamma-i\delta)\,dz\otimes\partial_{\bar z},
\ee
whence $f=\alpha+i\beta$ and $q=(\gamma-i\delta)\,dz^2$. Then at $z=0$,
\ben
\div A^{1,0}=(-e_1(\alpha)+e_2(\beta))e_1-(e_2(\alpha)+e_1(\beta))e_2
\ee
and
\ben
\div A^{0,1}=-(e_1(\gamma)+e_2(\delta))e_1+(e_2(\gamma)-e_1(\delta))e_2.
\ee
Computing $\bar\partial f=\partial_{\bar z}(\alpha+i\beta)\, d\bar z$ and $\bar\partial q=\partial_{\bar z} (\gamma-i\delta)\, d\bar{z}\otimes dz^2$ gives immediately the desired result.
\end{proof}

\begin{remark}
In particular, for a critical point $(g,\phi)$ the symmetric $(2,0)$-tensor associated with $A^{0,1}$ is a {\em tt-tensor}, that is, traceless and {\em transverse} (divergence-free). For $\gamma\geq2$, the previous lemma therefore recovers the standard identification of the space of tt-tensors with the tangent space of Teichm\"uller space given by holomorphic quadratic differentials. 
\end{remark}

We are now in a position to give an alternative characterisation of absolute minimisers if $\gamma\geq1$. The case of the sphere will be handled in \cref{min_sphere}.

\begin{proposition}\label{tt.tensor}
Let $\gamma\geq1$. The following statements are equivalent:
\begin{enumerate}[{\rm (i)}]
	\item\label{abs.min} $(g,\phi)$ is an absolute minimiser.
	\item\label{A.tt} $\nabla_X\phi=A(X)\cdot\phi$ for a traceless symmetric endomorphism $A$. 
	\item\label{beta.van} $(g,\phi)$ is critical and $\beta=0$.
\end{enumerate}
\end{proposition}

\begin{remark}
In particular, we recover the equivalence (ii)$\Leftrightarrow$(iii) of \cite[Theorem 13]{fr98} for the case $H=0$.
\end{remark}

\begin{proof}
By \cref{inf.E}, $(g,\phi)$ is an absolute minimiser if and only if $D\phi=0$. From~\eqref{diracpair} this is tantamount to $\tr A=0$, $\tr(A\circ J)=0$ and $\beta=0$. The trace conditions are equivalent to $A$ being symmetric and traceless whence the equivalence between \eqref{abs.min} and \eqref{A.tt}. Furthermore, \eqref{A.tt} immediately forces $\beta=0$. Conversely, \eqref{beta.van} together with the critical point equation in \cref{critical} implies $2A^tA=|A|^2\Id$, whence $2|A|^2=|K|$ by \cref{cor_gauss} \eqref{K.detA}. In particular, $2|A|^2=-K$ on the open set $U=\{x\in M_\gamma : K(x)<0\}$. Assume that $U$ is non-empty and not dense in $M_\gamma$, i.e.\ $\bar U \subset M_\gamma \setminus \{p\}$ for some $p \in M_\gamma$. Without loss of generality we may also assume $U$ to be connected. On its boundary the curvature vanishes so that in particular, $|A|=0$ on $\partial U$. Further, $|D\phi|^2=|A|^2+K/2=0$ on $U$ as a simple computation in an orthonormal frame using \cref{diracpair} reveals. As before, $D\phi=0$ implies that $A$ is traceless symmetric and divergence-free over $U$. In particular, $A$ corresponds to a holomorphic quadratic differential by \cref{div.free}. Since every holomorphic line bundle on the non-compact Riemann surface $M_\gamma \setminus \{p\}$ is holomorphically trivial (see for instance~\cite[Theorem 30.3]{fo81}), over $U$ the coefficients of $A$ arise as the real and imaginary part of a holomorphic function and are therefore harmonic. However, they are continuous on $\bar U$ and vanish on the boundary, hence $A=0$ by the maximum principle. In particular, $K=0$ on $U$, a contradiction. This leaves us with two possibilities. Either $U$ is dense in $M_\gamma$ or $U$ is empty. By Gau\ss{}-Bonnet the second case can only happen for genus $1$ and $g$ must be necessarily flat. In any case, $\phi$ is harmonic and therefore defines an absolute minimiser.
\end{proof}

\begin{corollary}\label{trichotomy}
Let $\gamma \geq 1$. If $(g,\phi)$ is an absolute minimiser, then $A$ is a tt-tensor. Furthermore, $K\equiv0$ if $\gamma=1$ and $K\leq0$ with only finitely many zeroes if $\gamma\geq2$.
\end{corollary}

\begin{remark}
\phantom{a}\hfill
\begin{enumerate}[(i)]
	\item As we will see in \cref{torus} there exist flat critical points which are not absolute minimisers.
	\item If $\gamma\geq2$ in \cref{tt.tensor} \eqref{beta.van}, it suffices to assume that $|\beta|=const$. Indeed, $(\beta\otimes\beta)_0$ induces a holomorphic section of $\kappa_\gamma^2$ and has therefore at least one zero. At such a zero, $(\beta\otimes\beta)_0=0$, hence $\beta=0$ in this point and thus everywhere.
	\item If $\beta=0$, then \cref{int.cond.Abeta} implies $\div(A\circ J)=0$. However, this does not yield an extra constraint as $\div(A\circ J)=\div A$ for $A$ symmetric.
\end{enumerate}
\end{remark}
%
%
%
\section{Critical points on the sphere}\label{sphere}
In this section we completely classify the critical points in the genus $0$ case where $M_\gamma$ is diffeomorphic to the sphere. In particular, up to isomorphism there is only one spin structure for $S^2$ is simply-connected.
%
\subsection{Twistor spinors}\label{twi.spi}
For a general Riemannian spin manifold $(M^n,\sigma,g)$ with spinor bundle $\Sigma_gM^n\to M^n$, a {\em Killing spinor} $\phi\in\Gamma(\Sigma_gM^n)$ satisfies
\ben
\nabla_X\psi=\lambda X\cdot\psi
\ee 
for any vector field $X\in\Gamma(TM)$ and some fixed $\lambda\in\C$, the so-called {\em Killing constant}. In particular, the underlying Riemannian manifold is Einstein with $\Ric=4\lambda^2 g$ so that $\lambda$ is either real or purely imaginary. If $M$ is compact and connected, only Killing spinors of {\em real} type, where $\lambda\in\R$, can occur \cite[Theorem 9 in Section 1.5]{bfgk91}. More generally we can consider {\em twistor spinors}. By definition, these are elements of the kernel of the {\em twistor operator} $T_g=\mr{pr}_{\ker\mu}\circ \nabla$, where $\mr{pr}_{\ker\mu}:\Gamma(T^*\!M\otimes\Sigma)\to\Gamma(\ker\mu)$ is projection on the kernel of the Clifford multiplication $\mu:T^*\!M\otimes\Sigma M\to\Sigma M$. Equivalently, a twistor spinor satisfies  
\ben
\nabla_X \phi=- \tfrac{1}{n} X \cdot D\phi
\ee
for all $X \in \Gamma(TM)$. The subsequent alternative characterisation will be useful for our purposes. The following proposition was stated (in slightly
different form) in~\cite{fr90}, see e.g.~\cite[Theorem 2 in Section 1.4]{bfgk91} for a proof.

\begin{proposition}\label{twistor}
On a Riemannian spin manifold $M^n$ the following conditions are equivalent:
\begin{enumerate}[{\rm (i)}]
	\item $\phi$ is a twistor spinor.
	\item $X \cdot \nabla_X \phi$ does not depend on the unit vector field $X$.
\end{enumerate}
\end{proposition}

\begin{example}\label{examp} 
(cf.~\cite[Example 2 in Section 1.5]{bfgk91}) On the round sphere $S^n$ there are Killing spinors $\psi_{\pm}\neq 0$ with $\lambda_\pm=\pm\frac12$. Furthermore, 
$\phi_{ab}=a\psi_++b \psi_-$ for constants $a,\,b\in\R$ are twistor spinors 
which are not Killing for $ab\neq 0$. Indeed, Killing spinors must have constant length, while $\phi_{ab}$ will have zeroes in general. If $n$ is even, then a spinor $\psi_+$ is a Killing spinor for 
the Killing constant $\tfrac{1}{2}$ if and only if $\psi_-:=\omega\cdot\psi_+$ is a Killing spinor for $-\frac12$. Moreover, if $n\equiv 2 \mod 4$, then these $\psi_\pm$ are pointwise orthogonal. In this particular case $\phi_{ab}=a\psi_++b \psi_-$ is a twistor spinor of constant length.
\end{example}

Using~\cite{habermann:90} the following lemma is straightforward.
 
\begin{lemma}\label{lem.twodim}
Let $(M_\gamma,\sigma)$ be a spin surface and $(g,\phi)\in\mc N$. Then the following conditions are equivalent.
\begin{enumerate}[{\rm(i)}]
	\item\label{enum.twist} $\phi$ is a twistor spinor.
	\item\label{enum.ab} There exist $a,b \in \R$ such that $\nabla_X \phi = a X \cdot \phi + b J(X) \cdot \phi$ for all $X \in \Gamma(TM)$.
	\item\label{enum.kill} There exist $\alpha\in\mR$ and a unit Killing spinor $\psi$ such that
\ben
\phi=\cos\alpha\;\psi+\sin\alpha\;\omega\cdot\psi.
\ee
\end{enumerate}
Furthermore, the Killing constant $\lambda$ of $\psi$ is given by $\lambda= \sqrt{a^2+b^2}$. 
\end{lemma}

\begin{remark}
Note that for~\eqref{enum.kill}, $\omega\cdot\psi$ is a Killing spinor with Killing constant $-\lambda$. 
\end{remark}

In terms of the associated pair $(A,\beta)$ we have $A=a\Id+bJ$ and $\beta=0$ for a twistor spinor. Hence \cref{lem.twodim} together with Proposition \ref{curvature} immediately implies:

\begin{corollary}
Let $\phi$ be a $g$-twistor spinor of unit length. Then $g$ has non-negative constant Gau{\ss} curvature $K=4(a^2+b^2)$. In particular, $K=0$ if and only if $\phi$ is a parallel spinor.
\end{corollary}

\begin{remark}
The previous corollary is a special case of \cite[Theorem~1, p. 69]{fr90} and 
\cite[Theorem~3, p. 71]{fr90} where it was shown to hold in any dimension.
\end{remark} 
%
\subsection{Critical points on the sphere}
Next we completely describe the set of critical points on the sphere.

\begin{theorem}\label{min_sphere}
On $M_0=S^2$, the following statements are equivalent:
\begin{enumerate}[{\rm (i)}]
	\item\label{enum.crit} $(g,\phi)$ is a critical point of $\mc E$.
	\item\label{enum.min}  $\mc E(g,\phi)=\pi$, i.e.\ $(g,\phi)$ is an absolute minimiser.
	\item\label{enum.twistor}  $\phi$ is a twistor spinor, i.e.
\begin{equation}\label{ab2}
\nabla_X \phi = a X \cdot \phi + b J(X) \cdot \phi
\end{equation}
for constants $a,b \in \R$. 
	\item\label{enum.killing}  There is a unit-length Killing spinor $\psi$ on $(S^2,g)$ and $\al\in\mR$
such that 
\begin{equation}\label{sum.kill}
\phi=\cos\al\;\psi+\sin\al\;\omega\cdot\psi
\end{equation}
\end{enumerate}
Moreover, any of these conditions implies that the Gau{\ss} curvature of $g$ is a positive constant.
\end{theorem}
\begin{proof}
Assume $(g,\phi)$ is a critical point. Since $H^1(S^2,\R)=0$, \cref{critical} and \cref{beta-is-harmonic} imply
\ben
\beta=0, \quad A^tA = \tfrac 12 |A|^2 \Id, \quad \div A=0,
\ee
whence $2|A|^2=|K|$. Since the set of points where $K<0$ cannot be dense on $S^2$ by Gau\ss{}-Bonnet, it must be empty (cf.\ the proof of \cref{tt.tensor}). In particular, $2|A|^2=K$. Since $|\nabla \phi|^2 = |A|^2$, Gau\ss{}-Bonnet again implies 
\ben
\mc E(g,\phi) = \tfrac 12 \int_M |\nabla \phi|^2 =\tfrac 12 \int_M |A|^2 = \tfrac 14 \int_M K = \tfrac 14 \cdot 4\pi = \pi
\ee 
Conversely, this implies that $(g,\phi)$ is critical by \cref{inf.E}.

Next assume that \eqref{enum.min} holds. The equality $2\pi = \int_M |\nabla \phi|^2$ gives the pointwise equality $|D\phi|^2 = 2 |\nabla \phi|^2$, cf.\ \eqref{dirnab} and \eqref{dirnab2}. On the other hand, equality in \eqref{dirnab} arises if and only if $e_1 \cdot \nabla_{e_1}\phi = e_2 \cdot \nabla_{e_2}\phi$. Multiplying with $\omega=e_1 \cdot e_2$ from the left yields the equation $e_1\cdot\na_{e_2}\phi=-e_2\cdot\na_{e_1}\phi$. Hence for $X=ae_1+be_2$ with $a^2+b^2=1$ we obtain 
\begin{align*}
X \cdot \nabla_X \varphi &= a^2 e_1\cdot\nabla_{e_1} \varphi + b^2 e_2 \cdot\nabla_{e_2} \varphi + ab (e_1 \cdot \nabla_{e_2}\varphi + e_2 \cdot \nabla_{e_2}e_1)\\
&=e_1 \cdot \nabla_{e_1} \varphi = e_2 \cdot \nabla_{e_2} \varphi.
\end{align*}
According to \cref{twistor}, $\phi$ is a twistor spinor .

The equivalence between \eqref{enum.twistor} and \eqref{enum.killing} follows directly from \cref{lem.twodim}.

Finally, \cref{sum.kill} states that $\phi$ is in the $S^1$-orbit of a Killing spinor which is clearly a critical point - its associated pair is $A=\lambda\id$ and $\beta=0$. Hence \eqref{enum.min} follows. 
\end{proof}

\begin{corollary}\label{critsphere}
Up to rescaling there is exactly one $\U(2)=S^1\times_{\Z_2}\SU(2)$ orbit of critical points on $S^2$.
\end{corollary}
%
%
%
\section{Critical points on the torus}\label{torus}
%
\subsection{Spin structures on tori}\label{spin.stru.tori}
Finally we investigate the genus $1$ case, that is we consider a torus $T^2_\Gamma=\R^2/\Gamma$ for a given lattice $\Gamma\subset\R^2$. Here, we have four inequivalent spin structures, three of which are bounding. In the case of a flat metric these can be described uniformly through homomorphisms $\chi:\Gamma\to\Z_2=\{-1,1\}=\ker\theta\subset\Spin(2)$ giving rise to an associated bundle $P_\chi:=\R^2\times_\zeta\Spin(2)$. Here, $\theta$ is the connected double covering $\Spin(2)\cong S^1\to \SO(2)\cong S^1$. The quotient map $\R^2\to T^2_\Gamma$ and the covering $\theta$ induce a map $\eta_\chi:P_\chi\to P_{\SO(2)}(T^2_\Gamma)$ which defines a spin structure. In fact, there is a bijection between $\Hom(\Gamma,\Z_2)\cong H^1(T^2_\Gamma;\Z_2)$ and isomorphism classes of spin structures on $T^2_\Gamma$ such that the non-bounding spin structure corresponds to the trivial homomorphism $\chi\equiv1$ (see~\cite{fr84} or \cite[Section 2.5.1]{ba81} for further details). For example, the non-bounding spin structure is the trivial spin structure given by $\id\times\theta:T^2\times\Spin(2)\to T^2\times\SO(2)$. Its associated spinor bundle is trivialised by parallel sections in contrast to the spinor bundles associated with the three bounding spin structures which do not admit non-trivial parallel spinors~\cite{hi74}. (Note that for flat metrics a parallel spinor is the same as a harmonic spinor in virtue of the Schr\"odinger-Lichnerowicz formula.) For an example of a bounding spin structure, consider the Clifford torus inside $S^3$. If we equip the resulting solid torus with the spin structure induced from its ambient $S^3$, then the induced spin structure on its boundary, i.e.\ the Clifford torus, is a bounding spin structure.
%
\subsection{Non-minimising critical points on tori}\label{nonmin.crit.pts}
We are going to show that on certain flat tori, critical points which are not absolute minimisers do exist. Examples, which are in fact saddle points, are provided by the following construction.

\medskip
 
We begin with two parallel unit spinors $\psi_1$ and $\psi_2$ on the Euclidean space $(\R^2,g_0)$ satisfying $\psi_1\perp\psi_2$ and $\psi_1\perp\omega\cdot\psi_2$. Then an orthonormal basis of the spinor module $\Delta$ is given by
$\{\psi_1,\omega \cdot \psi_1, \psi_2, \omega \cdot \psi_2 \}$. Thinking of $\omega$ as an imaginary unit, we set
$$
e^{t\omega}:=\cos(t)+\sin(t)\omega
$$ 
for $t\in\R$. In particular, the usual formul{\ae} such as $e^{(s+t)\omega}=e^{s\omega}e^{t\omega}$ or $\nabla e^{t\omega}=\omega e^{t\omega}$ hold. Furthermore, let $\alpha_1,\alpha_2 \in \R^{2*}$. For $\theta\in\R$ consider the unit spinor
\beq\label{crit.phi.ansatz}
\varphi(x) = \cos (\theta) e^{\alpha_1(x)\omega}\psi_1 + \sin(\theta) e^{\alpha_2(x)\omega}\psi_2
\ee
for which
\beq\label{cov_deriv}
\nabla_{(\cdot)}\varphi(x)=\cos(\theta)\alpha_1(\cdot)(x)\otimes e^{\alpha_1(x)\omega}\omega\cdot\psi_1 + \sin(\theta)\alpha_2(\cdot)(x)\otimes e^{\alpha_2(x)\omega}\omega\cdot\psi_2.
\ee
As both $\{e_1\cdot \psi_1,e_2\cdot \psi_1\}$ and $\{\psi_2,\omega\cdot\psi_2\}$ span the space orthogonal to $\psi_1$ and $\omega\cdot \psi_1$, there is a unit vector field $V$ such that $\psi_2=V\cdot \psi_1$. Parallelity of $\psi_1$ and $\psi_2$ imply parallelity of $V$. The pair $(A,\beta)$ corresponding to $\phi$ in the decomposition~\eqref{nabdec} is given by the $(1,1)$-tensor
\beq\label{A.crit.point}
A_x=\cos(\theta)\sin(\theta)(\alpha_2-\alpha_1)\otimes e^{(\alpha_1(x)+\alpha_2(x)+\pi/2)\omega}V
\ee
and the {\em parallel} $1$-form
\beq\label{beta.crit.point}
\beta = \cos^2 (\theta) \alpha_1 + \sin^2(\theta) \alpha_2.
\ee
In particular, we find $\det A=0$ in accordance with \cref{curvature} \eqref{K.A.beta}. Indeed, $\omega\cdot V=-V\cdot\omega$ and $Ve^{t\omega}=e^{-t\omega}V$ for $t\in\R$ so that
\ben
\varphi(x)=\bigl(\cos(\theta)e^{\alpha_1(x)\omega}+\sin(\theta)e^{\alpha_2(x)\omega}V\bigr)\psi_1
\ee
and
\beq\label{nabla.varphi}
\nabla_X\varphi(x)=\bigl(\cos(\theta)\alpha_1(X)e^{\alpha_1(x)\omega}-\sin(\theta) \alpha_2(X)e^{\alpha_2(x)\omega}V\bigr)\omega\cdot\psi_1.
\ee
On the other hand,
\ben
\bigl(\cos(\theta)e^{-\alpha_1(x)\omega} - \sin(\theta)e^{\alpha_2(x)\omega} V\bigr)\bigl(\cos(\theta)e^{\alpha_1(x)\omega} + \sin(\theta) e^{\alpha_2(x)\omega}V\bigr)=1,
\ee
and therefore
\ben
\psi_1 = \bigl(\cos(\theta) e^{-\alpha_1(x)\omega} - \sin(\theta) e^{\alpha_2(x)\omega}V \bigr)\varphi.
\ee
After substitution into~\eqref{nabla.varphi} this gives
\begin{align*}
\nabla_X\varphi = &\bigl( \cos^2(\theta) \alpha_1(X)+\sin^2(\theta)\alpha_2(X)\\ 
&+\cos(\theta)\sin(\theta)\big(\alpha_1(X)-\alpha_2(X)\big)e^{(\alpha_1(x)+\alpha_2(x))\omega}V\bigr)\omega\cdot\varphi\\
=&\cos(\theta)\sin(\theta)\big(\alpha_2(X)-\alpha_1(X)\big) e^{(\alpha_1(x)+\alpha_2(x)+\pi/2)\omega}V\cdot\varphi\\
&+\big(\cos^2(\theta) \alpha_1(X) + \sin^2(\theta) \alpha_2(X)\big)\omega\cdot\varphi.
\end{align*}

\medskip

Next we compute the negative gradient of $\mc E$ in $(g,\phi)$. This is most easily done by considering the identities in~\eqref{neggrad} from which $4Q_1(g,\varphi)=-|\nabla \varphi|^2g-\div T_{g,\varphi}+2\langle \nabla \varphi \otimes \nabla \varphi \rangle$. Using~\eqref{cov_deriv} we compute
\ben
\langle\nabla\varphi\otimes\nabla\varphi\rangle=\cos^2(\theta)\alpha_1\otimes\alpha_1+\sin^2(\theta)\alpha_2
\otimes\alpha_2.
\ee
Since 
$$
|\nabla \varphi|^2 = \tr \langle \nabla \varphi \otimes \nabla \varphi \rangle= \cos^2(\theta) |\alpha_1|^2 + \sin^2(\theta) |\alpha_2|^2
$$ 
we obtain
\ben
\tfrac 12 \langle \nabla \varphi \otimes \nabla \varphi \rangle- \tfrac 14 |\nabla \varphi|^2g = \tfrac 12 \langle \nabla \varphi, \nabla \varphi \rangle_0 = \tfrac{1}{2}\cos^2(\theta) (\alpha_1 \otimes \alpha_1)_0 + \tfrac{1}{2}\sin^2(\theta)(\alpha_2 \otimes \alpha_2)_0.
\ee
Finally, if $\{e_1,e_2\}$ is the standard basis of $\R^2$, then as in~\eqref{div.T.comp}
\ben
\div T_{g,\varphi}(e_1,e_1) = e_2( \omega \varphi, \nabla_{e_1} \varphi \rangle = e_2(\beta(e_1)) = 0
\ee
since $\beta$ is parallel. Next $Q_2(g,\varphi)=- \nabla^*\nabla \varphi + |\nabla \varphi|^2 \varphi$ by~\eqref{neggrad}. Again using \cref{cov_deriv} we compute
\ben
\nabla^*\nabla \varphi=|\alpha_1|^2\cos(\theta)e^{\alpha_1(x)\omega}\psi_1+|\alpha_2|^2 \sin(\theta)e^{\alpha_2(x)\omega}\psi_2.
\ee
Altogether we get for the spinor $\phi$ defined by~\eqref{crit.phi.ansatz} that
\begin{align*}
Q_1(g,\varphi) =& \tfrac{1}{2}\cos^2(\theta) (\alpha_1\otimes\alpha_1)_0 + \tfrac{1}{2}\sin^2(\theta) (\alpha_2 \otimes \alpha_2)_0,\\
Q_2(g,\varphi) =&-|\alpha_1|^2 \cos(\theta) e^{\alpha_1(x)\omega} \psi_1 - |\alpha_2|^2 \sin(\theta) e^{\alpha_2(x)\omega}\psi_2\\
&+\big( \cos^2(\theta)  |\alpha_1|^2 + \sin^2(\theta) |\alpha_2|^2\big)\varphi.
\end{align*}
For a critical point we need $Q_1$ and $Q_2$ to vanish. Now $Q_1(g,\varphi)$ vanishes if and only if $\cos^2(\theta) \alpha_1 \otimes \alpha_1 + \sin^2(\theta) \alpha_2 \otimes \alpha_2$ is a constant multiple of the Euclidean metric~$g$. This in turn is the case if and only if $\alpha_1 \perp \alpha_2$ and $|\cos (\theta)|\, |\alpha_1| = |\sin(\theta)|\, |\alpha_2|$. Furthermore, $Q_2(g,\varphi) = 0$ if and only if $\nabla^*\nabla\varphi = f \varphi$ for some function $f:\R^2\to\R$, i.e.\ if
\begin{align*}
|\alpha_1|^2 \cos(\theta)e^{\alpha_1(x)\omega}\psi_1 &+|\alpha_2|^2 \sin(\theta) e^{\alpha_2(x)\omega}\psi_2\\
=&f(x)\big(\cos(\theta)e^{\alpha_1(x)\omega}\psi_1+\sin(\theta)e^{\alpha_2(x)\omega}\psi_2\big).
\end{align*}
Again this holds if and only if $|\alpha_1|^2=f(x)=|\alpha_2|^2$. 

\medskip

Summarising, the spinor $\varphi$ in~\eqref{crit.phi.ansatz} is a critical point if and only if 
\begin{equation}\label{conditions}
\alpha_1\perp\alpha_2,\quad|\alpha_1|=|\alpha_2|,\quad(\theta - \pi/4)\in (\pi/2)\Z 
\end{equation}
are satisfied. When does then $\varphi$ descend to a well-defined spinor on a torus? For $\ell:=\pi/|\alpha_1|$ consider first the square torus $T_\ell:=\R^2/\Gamma_\ell$ whose lattice is spanned by
\begin{equation}\label{lattice}
\gamma_1=\ell\begin{pmatrix}1\\1\end{pmatrix} \quad \text{and} \quad \gamma_2=\ell\begin{pmatrix}1\\-1\end{pmatrix}.
\end{equation}
Possibly after an additional rotation we may assume without loss of generality that $\alpha_i=|\alpha_1|e^i$ for the standard basis $(e^1,e^2)$ of $\R^{2*}$. If $\sigma_\chi$ is the (necessarily bounding) spin structure defined by the group morphism $\chi_\ell:\Gamma_\ell\to\Z_2$, $\chi_\ell(\gamma_1)=\chi_\ell(\gamma_2)=-1$, then 
\begin{equation}\label{compat} 
e^{\alpha_1(\gamma)\omega}= e^{\alpha_2(\gamma)\omega}=\chi_\ell(\gamma)
\end{equation}
so that $\varphi$ descends to $(T^2_\ell,\sigma_\ell)$ and gives rise to a critical point there. More generally, $\varphi$ descends to any covering $T_\Gamma=\R^2/\Gamma$ of $T_\ell$, where the spin structure on $T_\Gamma$ is induced by $\chi=\chi_\ell|_\Gamma$. For instance, the double covering $T_{2\ell}\to T^2_\ell$ yields a square torus for which $\varphi$ descends to a spinor with respect to the non-bounding spin structure defined by $\chi\equiv1$. Conversely, any torus~$T_\Gamma$ to which $\varphi$ descends is necessarily a covering of $T^2_\ell$. Indeed, assume~\eqref{compat} holds for the spin structure $\sigma_\chi$ on $T^2_\Gamma$ instead of $\sigma_\ell$ on $T^2_\ell$, and let $\Gamma_0=\ker\chi$. In particular, $\Gamma_0\subset 2\ell\Z^2$. If $\sigma_\chi$ is the non-bounding structure, then $\chi\equiv1$ and therefore $\Gamma_0=\Gamma$. Otherwise, there exists $\gamma_0\in\Gamma$ with $\chi(\gamma_0)=-1$ so that~\eqref{compat} implies
$$
\gamma_0-\frac\ell2 \begin{pmatrix}1\\1\end{pmatrix}\in \ell\mZ^2.
$$
In particular, $\Gamma$ is contained in $\Gamma_\ell$. 

\begin{remark}\label{A.beta.crit.point}
From \cref{A.crit.point} and \cref{beta.crit.point} it follows immediately that for a flat critical point on the torus, $\beta^\sharp\in\ker A$. Conversely, any critical point satisfying this condition is necessarily flat, cf.\  \cref{class.crit.point}.
\end{remark}

In conclusion we established the existence of critical points $(g_0,\varphi_0)$ on any torus $T_\Gamma$ covering $T_\ell$. Its spin structure is determined by the restriction of $\chi_\ell$ to $\Gamma$. Finally, we will show the existence of saddle points. First of all, if $\varphi$ satisfies \eqref{compat}, then
$$
\mc E(g,\varphi)=\frac{\cos^2(\theta)|\alpha_1|^2+\sin^2(\theta) |\alpha_2|^2}{2} \operatorname{area}(T^2_\ell).
$$
Now $\operatorname{area}(T^2_\ell)=2\ell^2$, and if $(g_0,\varphi_0)$ is critical, $\cos^2(\theta_0)=\sin^2(\theta_0)=\frac 12$ and $|\alpha_1|=|\alpha_2|=\pi/\ell$, whence $\mc E(g_0,\varphi_0)=\pi^2$. Next we construct special curves $(g_t,\varphi_t)$ through $(g_0,\varphi_0)$. The metric $g_t$ is obtained through an area-preserving deformation of $T^2_\ell$ by taking the lattice $\Gamma_t$ spanned by
$$
\gamma_1(t)=\frac{\ell}{2}\begin{pmatrix}1\\1\end{pmatrix}(1+t) \quad \text{and} \quad \gamma_2(t)=\frac{\ell}{2}\begin{pmatrix}1\\-1\end{pmatrix}\frac{1}{(1+t)}. 
$$
The spinor will be modified through $\theta=\theta(t)=ct+\theta_0$. Then $\mc E(g_t,\varphi_t)=f(t)\pi^2$
for 
$$
f(t) = \cos^2 (\theta(t)) \frac{1}{(1+t)^2}  + \sin^2 (\theta(t)) (1+t)^2.
$$
Since $f''(0)=8\theta'(0)+4$, the second derivative takes any real value by suitably choosing the slope $c$ in $\theta(t)$. The non-minimising critical point on $T^2_0$ is therefore a saddle point, and so are the critical points obtained by taking covers. 
%
\subsection{Classification of flat critical points on the torus}
We are now in a position to classify the flat critical points on the torus. Recall the decomposition $A=A^{1,0} + A^{0,1}$, where $A^{1,0}$ is the trace and skew-symmetric part of $A$, while $A^{0,1}$ is the symmetric part of $A$. If $A$ is associated with a critical point, these components correspond to a holomorphic function $f$ and a quadratic differential $q$, cf.\ \cref{div.free}. From the coordinate description of~\eqref{deco.endo} one easily verifies the identities 
$$
\det A=\det A^{1,0}+\det A^{0,1}
$$
and
\begin{align*}
(A^{1,0})^tA^{1,0}&=\det A^{1,0}\cdot\id, &
(A^{0,1})^2 &=-\det A^{0,1}\cdot\id,\\
A^{0,1}A^{1,0}&= (A^{1,0})^t A^{0,1}, & A^{1,0}A^{0,1} &= A^{0,1} (A^{1,0})^t.
\end{align*}
In particular, these identities imply
\ben
(A^tA)_0 = 2(A^{1,0})^tA^{0,1} = 2 A^{0,1}A^{1,0},
\ee
so that $(A^tA)_0$ corresponds to the quadratic differential $2fq$.

\begin{theorem}\label{classification}
A flat critical point on the torus is either an absolute minimiser, i.e.\ a parallel spinor, or a non-minimising critical point as in \cref{nonmin.crit.pts}.
\end{theorem}
\begin{proof}
Let $(g,\varphi)$ be a critical point on $M=T^2$ with vanishing Gau{\ss} curvature and associated pair $(A,\beta)$. The Euler-Lagrange equation implies
\begin{gather*}
d\beta = \div \beta=0\\
\div A = \div A^{1,0} + \div A^{0,1} = 0.
\end{gather*}
Furthermore, together with $K=0$ and \cref{cor_gauss},
\ben
\div (A^tA)_0 =0.
\ee
On $M=T^2$ we may trivialize $T^*\!M^{1,0}$ and write $q = h \, dz^2$ for $h = c -id$ globally. Then $\div A=0$ yields the equation $\partial_{\bar z} f + \partial_z \bar h =0$. The traceless symmetric endomorphism $(A^tA)_0$ corresponds to the quadratic differential $fq = fh \, dz^2$ and $\div(A^tA)_0=0$ yields the holomorphicity of $fh$. In particular we get $fh=c$ for some constant $c\in\C$.
Moreover, by \cref{cor_gauss} again, $K=0$ also yields
\ben
\det A = \det A^{1,0} + \det A^{0,1} =0.
\ee
Consequently, $|f|^2=|h|^2=|c|$, for $\det A^{1,0} = |f|^2$ and $\det A^{0,1} = -|h|^2$. Rotating the coordinate system if necessary we may assume that $c$ is a non-negative real number. If $c=0$, then $A=\beta=0$ and we have an absolute minimiser, so assume from now on that $c>0$. We want to show that $(g,\phi)$ is of the form of the critical points in \cref{nonmin.crit.pts}. Scaling the metric on the spinor bundle appropriately we may assume that $c=1/4$. Writing $f(x,y)=e^{i\tau(x,y)}/2$ for $\tau:T^2 \to\R$, we have $h=\bar f$ and $\partial_{\bar z}f+\partial_z\bar h =0$ if and only if $\partial_{\bar z}\tau+\partial_z\tau=\partial_x\tau=0$, whence $\tau\equiv\tau(y)$. It follows that
\ben
A=\begin{pmatrix}
\cos\tau(y) & 0\\
\sin\tau(y) & 0
\end{pmatrix}.
\ee
Further, $\beta$ is parallel since $K=0$ and $\beta$ is harmonic. Since
$$
(A^tA +\beta\otimes\beta)_0=0
$$
we obtain $\beta=dy$. The integrability condition of \cref{int.cond.Abeta} now reads
\ben
\nabla_{\partial_y}A(\partial_x)=2\beta(\partial_y)J(A(\partial_x)) 
\ee
which implies $\theta'(y) = 2\beta(\partial_y) =2$. Hence $A$ and $\beta$ are as in \cref{A.crit.point} and \cref{beta.crit.point} for $V=\partial_y$, $\theta=\pi/4$, $\alpha_1=dx+dy$ and $\alpha_2=dy-dx$.
\end{proof}

\begin{remark}
As we have remarked in \cref{spin.stru.tori}, non-trivial parallel spinors only exist for the non-bounding spin structure. However, by the previous proposition, flat critical points can also exist for the bounding spin-structures.
\end{remark}

It remains an open question if a critical point on the torus is necessarily flat, but at least we can give a number of equivalent conditions.

\begin{proposition}\label{class.crit.point}
For a critical point $(g,\phi)$ on the torus which is associated with $(A,\beta)$, the following conditions are equivalent. 
\begin{enumerate}[{\rm (i)}]
	\item $g$ is flat.
	\item $|\beta|=const$.
	\item $\beta^\sharp\in\ker A$.
\end{enumerate}
Moreover, any of these conditions implies 
\beq\label{Anorm.betanorm}
|A|^2 = |\beta|^2.
\ee
\end{proposition}
\begin{proof}
If $(g,\varphi)$ is a flat critical point, then $\beta$ is parallel and hence has constant length. Conversely, if $(g,\phi)$ is critical with $|\beta|=const$, then the Weitzenb\"ock formula on $1$-forms~\eqref{weitzenboeck} and Gau{\ss}-Bonnet immediately imply that $\nabla\beta=0$. Therefore, $(A^tA+\beta\otimes\beta)_0=0$, whence $\div(A^tA)_0=0$ by \cref{beta-is-harmonic} \eqref{div.beta.beta}. But either $\beta\equiv0$ so that $\nabla\phi=0$ by~\cref{tt.tensor}, or $\beta$ has no zeroes at all and we can apply \cref{cor_gauss}. In both cases it follows $K=0$.

On the other hand, for a flat critical point $(g,\varphi)$, $\beta^\sharp\in\ker A$ follows from \cref{A.beta.crit.point}. Conversely, let $\beta^\sharp\in\ker A$. If $\beta(x)=0$, then~\eqref{weitzenboeck} implies $\nabla^*\nabla\beta(x)=0$. Otherwise, $\beta(x)$ is a non-trivial element in the kernel of $A$ so that $\det A(x)=K(x)=0$ by \cref{cor_gauss} \eqref{K.detA}. Again, we find $\nabla^*\nabla\beta(x)=0$ so that $\beta$ is actually parallel. Then either $\beta\equiv0$ and $g$ is flat (for in this case $(g,\phi)$ is an absolute minimiser), or $\beta$ is nowhere vanishing so that $K=\det A\equiv0$.

If any of these equivalent conditions holds, then $(A^tA + \beta \otimes \beta)_0=0$ and $\beta^\sharp \in \ker A$, whence
\begin{align*}
0=\langle\beta^\sharp,(A^tA + \beta \otimes \beta)_0 \beta^\sharp \rangle&=|A (\beta^\sharp)|^2 +|\beta|^4 - \tfrac 12 \tr (A^tA) |\beta|^2 - \tfrac 12 |\beta|^4\\
&=\tfrac 12|\beta|^2\big(|\beta|^2-|A|^2\big).
\end{align*}
This implies $|A|=|\beta|$ or $\beta=0$. In the latter case \cref{tt.tensor} implies $A=0$.
\end{proof}
\bigskip
\paragraph{\bf Acknowledgements.} The authors thank T.\ Friedrich for providing interesting references and the referee for carefully reading the manuscript.
%


\begin{thebibliography}{99}
\bibitem{ammann.dahl.humbert:09}
{\sc Ammann, B., Dahl, M., and Humbert, E.},
{\em Surgery and harmonic spinors},
\newblock {Adv. Math.} {\bf 220} (2009), 523--539.

\bibitem{amwewi12}
{\sc B.\ Ammann, H.\  Wei{\ss}, and F.\ Witt},
{\em A spinorial energy functional: critical points and gradient flow},
Preprint 2012, \href{http://www.arxiv.org/abs/1207.3529}{arXiv:1207.3529}.

\bibitem{at71}
{\sc M.\ Atiyah},
{\em Riemann surfaces and spin structures},
Ann.\ Sci.\ \'Ecole Norm.\ Sup.\ (4) {\bf4} (1971), 47--62. 

\bibitem{baer:98}
{\sc C. B\"ar}, 
\emph{Extrinsic bounds for eigenvalues of the {D}irac operator}, Ann.
Global Anal. Geom. \textbf{16} (1998), 573--596.

\bibitem{baer.schmutz:92}
{\sc C. B\"ar and P. Schmutz}, 
\emph{Harmonic spinors on {R}iemann surfaces}, Ann.
Global Anal. Geom. \textbf{10} (1992), 263--273.

\bibitem{ba81}
{\sc H.\ Baum},
{\em Spin-Strukturen und Dirac-Operatoren \"uber pseudoriemannschen Mannigfaltigkeiten}, 
Teubner-Texte zur Mathematik, {\bf 41}. BSB B.\ G.\ Teubner, Leipzig, 1981. 

\bibitem{bfgk91}
{\sc H.\ Baum, T.\ Friedrich, R.\ Grunewald and I.\ Kath},
{\em Twistors and Killing spinors on Riemannian manifolds},
Teubner-Texte zur Mathematik, {\bf124}, B.\ G.\ Teubner, Stuttgart, 1991. 

\bibitem{fo81}
{\sc O.\ Forster}, 
{\em Lectures on Riemann surfaces}, 
Graduate Texts in Mathematics, {\bf81}, Springer-Verlag, Berlin, 1981. 

\bibitem{fr84}
{\sc T.\ Friedrich},
{\em Zur Abh\"angigkeit des Dirac-Operators von der Spin-Struktur},
Colloq.\ Math.\ {\bf48} (1984), no.\ 1, 57--62.

\bibitem{fr90}
{\sc T.\ Friedrich},
{\em On the conformal relation between twistors and Killing
spinors}, Suppl. Rend. Circ. Mat. Palermo, II. Ser. 22, 59--75 (1990).
Avaliable on \url{https://eudml.org/doc/220945}

\bibitem{fr98}
{\sc T.\ Friedrich},
{\em On the spinor representation of surfaces in Euclidean $3$-space},
J.\ Geom.\ Phys.\ {\bf28} (1998), no.\ 1-2, 143--157. 

\bibitem{fr00}
{\sc T.\ Friedrich},
{\em Dirac operators in Riemannian geometry},
Graduate Studies in Mathematics {\bf 25}, AMS, Providence, 2000.

\bibitem{gigr10}
{\sc N.\ Ginoux, J.-F. Grosjean},
{\em Almost harmonic spinors}, 
C.\ R.\ Math.\ Acad.\ Sci.\ Paris {\bf 348} (2010), 811--814.

\bibitem{gu66}
{\sc R.\ Gunning},
{\em Lectures on Riemann surfaces},
Princeton Univ. Press, New Jersey, 1966.

\bibitem{habermann:90}
{\sc K. Habermann}, 
{\em The twistor equation of Riemannian manifolds},
J. Geom. Phys. {\bf 7}, 469--488 (1990).

\bibitem{hi74}
{\sc N.\ Hitchin},
{\em Harmonic spinors},
Adv.\ in Math. {\bf14} (1974), 1--55. 

\bibitem{kenmotsu:79} {\sc K.~Kenmotsu},
{\em Weierstrass formula for surfaces of prescribed mean curvature},
Math. Ann. {\bf 245} (1979), 89--99.

\bibitem{kusner.schmitt:p95}
{\sc R.~Kusner and N.~Schmitt,}
\emph{The spinor representation of minimal surfaces}, 
preprint, http://www.arxiv.org/abs/dg-ga/9512003, 1995.

\bibitem{lami89}
{\sc H.\ Lawson and M.-L.\ Michelsohn},
{\em Spin geometry}, 
Princeton Univ. Press, New Jersey, 1989.

\bibitem{martens:68}
{\sc H. Martens},
{\em Varieties of special divisors on a curve. {II}},
J. Reine Angew. Math. {\bf 233} (1968), 89--100.

\bibitem{mi65}
{\sc J.\ Milnor},
{\em Remarks concerning spin manifolds}, 
1965 Differential and Combinatorial Topology (A Symposium in Honor of Marston Morse) pp. 55--62 Princeton Univ. Press, New Jersey.

\bibitem{schmitt:diss}
{\sc N.~Schmitt}, \emph{Minimal surface with planar embedded ends}, {P}h.{D}.\ dissertation, University of Amherst, 1993.

\bibitem{trautman:92}
{\sc A.~Trautman}, \emph{Spinors and the Dirac operator on hypersurfaces. I. General theory},
J. Math. Phys. {\bf 33} (1992), 4011--4019.

\bibitem{trautman:95}
{\sc A.~Trautman}, \emph{The Dirac operator on hypersurfaces},
Acta Phys. Polonica B {\bf 26} (1995), 1283--1310.

\bibitem{tr08}
{\sc M.\ Traizet},
{\em On the genus of triply periodic minimal surfaces},
J.\ Differential.\ Geom. {\bf 79} (2008), no.\ 2, 243--75.
\end{thebibliography}
\end{document}